\documentclass[10pt]{amsart}
\usepackage{amsmath,amssymb,amsthm,graphicx,tikz}
\usepackage{mathrsfs}
\usepackage{comment}
\usepackage{enumerate}
\usepackage[colorlinks=true]{hyperref}
\usepackage{xcolor}
\usepackage{cases}

\title[Response Solution for dissipative systems]
{Response solutions to quasi-periodically forced systems, even to possibly ill-posed PDEs, with strong dissipation and any frequency vectors}
\author[F.Wang]{Fenfen Wang}
\address{
	School of Mathematics, Shandong University,
	Jinan, Shandong 250100, P.R.China.
}
\email{ffenwang@hotmail.com}
\thanks{F.W. is supported by CSC  by the National Natural Science
	Foundation of China (Grant Nos. 11171185,10871117), F.W.
	thanks G.T. for hospitality 2018-2020}
\author[R.de la Llave]{Rafael de la Llave}
\address{School of Mathematics, Georgia Inst. of Technology,
	Atlanta GA, 30332, USA}
\email{rafael.delallave@math.gatech.edu}
\thanks{R. L. is supported in part by NSF grant DMS 1800241. }
\thanks{ This material is based upon work supported by the National Science Foundation under Grant No. DMS-1440140 at MSRI during Fall 2018.}
\date{\today}

\newtheorem{theorem}{Theorem}
\newtheorem{meta-thm}[theorem]{Meta-Theorem}
\newtheorem{lemma}[theorem]{Lemma}

\newtheorem{proposition}[theorem]{Proposition}

\newtheorem{remark}[theorem]{Remark}
\newtheorem{definition}[theorem]{Definition}

\numberwithin{equation}{section}

\begin{document}

\maketitle

\begin{abstract}

We consider several models (including
both multidimensional ordinary differential equations (ODEs) and partial differential
equations (PDEs), possibly ill-posed), subject to very strong damping
and quasi-periodic external forcing. We study the existence of
response solutions (i.e., quasi-periodic solutions with the same
frequency as the forcing). Under some regularity assumptions on the
nonlinearity and forcing, without
any arithmetic condition on the forcing frequency $\omega$, we show
that the response solutions indeed exist. Moreover, the solutions we obtained possess optimal regularity in
$\varepsilon$ (where $\varepsilon$ is the inverse of the coefficients
multiplying the damping) when we consider $\varepsilon$ in a domain that does not include the
origin $\varepsilon=0$ but has the origin on its boundary. We get that the response solutions depend continuously on
$\varepsilon$ when we consider $\varepsilon $ tends to $0$.  However, in general, they may not be differentiable at $\varepsilon=0$. In this paper, we allow
multidimensional systems and we do not require that the unperturbed
equations under consideration are Hamiltonian.
One advantage of the  method in the present paper is that it gives results for  analytic,
finitely differentiable and low regularity forcing and nonlinearity, respectively. As a matter of fact,
we do not even need that the forcing is continuous. Notably, we obtain
results when the forcing is in $L^2$ space and the nonlinearity is just
Lipschitz as well as in the case that the forcing is in $H^1$ space and the
nonlinearity is $C^{1 + \text{Lip}}$.
In the proof of our results, we reformulate the existence of
response solutions as a fixed point problem in appropriate 
spaces of smooth functions. 
\end{abstract}



\section{Introduction}\label{sec:intro}

In recent times, there has been much interest in the study of response
solutions for nonlinear mechanical models subject to strong
dissipation and quasi-periodic external forcing. We recall that response solutions are solutions with the same frequency as the forcing. The mechanical systems are
second order equations. Since the large coefficients of
dissipation 
are factors of terms involving the first  derivative, 
this is a singular perturbation.
\medskip

We are interested
in finding response solutions for two kinds of equations. We first
consider an ODE model of the form:
\begin{equation}\label{000.0}
\begin{split}
 x_{tt}+\frac{1}{\varepsilon} x_t+ g(x)= f(\omega t),\,\, x\in \mathbb{R}^n.
\end{split}
\end{equation}
The equation \eqref{000.0} is referred as \emph{``varactor''} equations in the literature \cite{Rafael13,Gentile10,  Gentile14,Guido17,Gentile17}.
\medskip

We also consider PDE models. One particular example is obtained from the Boussinesq equation (derived in the paper \cite{Bou72}) by adding a singular friction
proportional to the velocity:
\begin{equation}\label{b-e}
\begin{split}
u_{tt}+\frac{1}{\varepsilon}u_t-\beta u_{xxxx}-u_{xx}=(u^{2})_{xx}+f(\omega t,x), \,\,x\in\mathbb{T}=\mathbb{R}/2 \pi \mathbb{Z},\,\,
 \beta>0,
\end{split}
\end{equation}
where $\beta>0$ is a  parameter.  Of course, the equation \eqref{b-e} will be supplemented with periodic boundary conditions.
We note that the positive sign of $\beta$ makes equation \eqref{b-e} ill-posed. That is, there are many initial conditions that do not lead to solutions.  It is, however,
possible that there is a systematic way to construct many special solutions, for some ill-posed Boussinesq equations, which are physically observed (we refer to the papers \cite{Rafael16, Llave09, Honngyu1, Honngyu2}). 
\medskip

In both equations \eqref{000.0} and \eqref{b-e}, $\varepsilon $ is a
small parameter in $\mathbb{R}$ and $\omega \in \mathbb{R}^d$ with
$d\in \mathbb{N}_{+}:=\mathbb{N}\setminus\{0\}$. The forcing $f$ is quasi-periodic with respect
to time $t$. Note that in the PDE \eqref{b-e}, the forcing may depend on the
space variable. At this moment, we think of the forcing as a quasi-periodic function
taking values in a space of functions.  

In equation \eqref{000.0},
one considers the nonlinearity $g$ as a function
from $\mathbb{R}^n$ to $\mathbb{R}^n$ with $n\in \mathbb{N}_{+}$
and the forcing $f$ as a function from $ \mathbb{T}^d$ to $
\mathbb{R}^n$.  We will obtain several results depending on the
regularity assumed for $f$ and $g$. First, we will consider that the functions $f$ and $g$ are
real analytic such that they take real values for real arguments,
which are what appears in physical applications, with $\varepsilon \in
\mathbb{R}$.  We will also consider highly differentiable
functions $f$ and $g$, such as $f\in H^{m}(m>\frac{d}{2})$ and $g$ is
$C^{m+l},\,l=1,2,\cdots$.  In addition, we will obtain results for
rather irregular functions $f$ and $g$. For example, the forcing $f$ is in the
$L^2$ space, the nonlinearity $g$ is just Lipschitz or $f$ is in the
$H^1$ space, $g$ is $C^{1 + \text{Lip}}$.

In equation \eqref{b-e}, we consider the function $f:\mathbb{T}^d\times \mathbb{T}\rightarrow
\mathbb{R}$. Analogously to the case of
\eqref{000.0}, we will present results for $f$ being real analytic and
finitely differentiable with high regularity.  
Note that in the study of the PDE model \eqref{b-e}, we will just focus in the physically relevant case of 
a specific nonlinearity $(u^{2})_{xx}$.
It is possible to discuss general nonlinearities in a regularity class, but being
unaware of a physical motivation, we leave these generalizations to the readers. 
We emphasize that, in \eqref{b-e}, the nonlinearity $(u^{2})_{xx}$  is unbounded from 
one space to itself, but the fixed point problem we consider 
overcomes  this problem since there will be smoothing factors.  

From the point of physical view, the parameter $\varepsilon$ is real. However, it is natural to consider $\varepsilon$ in a complex domain when we consider our problem in an analytic setting.  It is important to notice that the complex domain we use does not include the origin but
accumulates on it. Indeed,  the solutions fail to be
differentiable at $\varepsilon = 0$ in the generality
considered in the present paper (see Remark~\ref{noo-de}). However,  we will show that the response solutions depend continuously on
$\varepsilon$ as $\varepsilon$ tends to $0$.
\subsection{Some remarks on the literature}
The problem of the response solutions for dissipative systems has been
studied by several methods. One method is based on developing
asymptotic series and then show that they can be resummed using
combinatorial arguments, which are established using the so-called
\emph{``tree formalism''}. This can be found in the literature
\cite{Gentile05, Gentile06,GG10,GGO10}. Recent papers developing this method  are \cite{Guido17,Gentile17}. We point out that one important novelty of
the papers \cite{Guido17, Gentile17} is that no arithmetic condition
is required in the frequency of the forcing. A later
method is to reduce the existence of response solutions to a fixed
point problem, which is analyzed in a ball in an appropriate Banach
space, centered in the solution predicted by the asymptotic
expansion. In this direction, we refer to \cite{Rafael13,Rafael17} and
references there. Note that  the papers
\cite{Rafael13,Rafael17} considered the perturbative expansion to
low orders on $\varepsilon$ and obtains a reasonably approximate
solutions in a neighborhood of $\varepsilon=0$.  Nevertheless, to
obtain the asymptotic expansions, one needs to solve equations
involving small divisors and assume some non-degeneracy
conditions. Note that the small divisors assumed in \cite{Rafael13,Rafael17} are weaker than the Diophantine conditions in KAM theory. In this paper, we will not assume any small divisors conditions since we do not attempt to get
the approximate solution through an asymptotic expansion.

Since the literature is growing, it is interesting to compare
systematically results.  There are several figures of merit for
results on the existence of response solutions.
 \begin{enumerate}
 \item 
The arithmetic properties required in the external forcing
frequency, such as Diophantine condition, Bryuno condition, or
even weaker conditions, etc.

 \item   
The analyticity domain in $\varepsilon$ established. Since we do not expect that the asymptotic series converges, this
domain does not include a ball centered at the origin. Note that the
shape of this analyticity domain is very important to study properties 
of the asymptotic series. For example, Borel summability in  \cite{Gentile05,Gentile06}.
In the generality we consider in this paper, the solutions we construct fail to be even differentiable
at the origin $\varepsilon=0$. (See Remark~\ref{noo-de}). 
\item 
Whether the method gives  some asymptotic expansions for the solutions. 
\item 
Whether the method can deal with the forcing function $f$ 
which has low regularity 
(e.g. $f\in L^2$ or $f\in H^1$) and the nonlinearity function $g$ of 
low regularity (the case of piecewise differentiable functions 
appears in some applications). 
\item 
The generality of the models considered (e.g. whether the method requires that the system is 
Hamiltonian, Reversible, etc.)  
\item 
Smallness conditions imposed on functions $f$ and $g$.
\end{enumerate}

Notice that all these figures of merit cannot be accomplished at the
same time. Obtaining more conclusions on the solutions (e.g. the
existence of asymptotic expansions) will require more regularity and
some arithmetic conditions on the frequency. 
\subsection{The method in the present paper}
 From the
strictly logical point of view,  our paper and
\cite{Guido17,Gentile17} are completely different even if they are
motivated by the same physical problem for the model \eqref{000.0}. More precisely,
the present paper deals with not only  analytic problems but also 
finitely differentiable problems and even just Lipschitz problems by the method of fixed point theorem. In contrast, the papers  \cite{Guido17,Gentile17} apply resummation methods to
establish the existence of response solutions under analytic condition.  
 
In the multidimensional case in equation \eqref{000.0}, compared with  \cite{Guido17}, the methods presented in this paper do not
need that the oscillators without dissipation  are Hamiltonian or
that the linearization of $g$ at the origin (i.e. $Dg(0)$, which is a $n\times n$ matrix) is positive definite. Further, we
do not assume that the matrix  $Dg(0)$ is diagonalizable or symmetric. We allow Jordan blocks that appear naturally in problems at resonance \cite{gazzo15,gazz15}.

However,  we note that our method for analytic case involves smallness assumptions in the forcing $f$ but not in the nonlinear part  $\hat{g}$ of $g$. In the case of $L^2$ and $H^1$, we involve just smallness assumptions on $\hat{g}$ but not $f$. For the highly differentiable case (i.e. $H^{m},\,m>\frac{d}{2}$), we choose either smallness assumption for $f$ or $\hat{g}$. (See Section~\ref{sec:small}).

 As a further application, we consider adding dissipative terms 
to the Boussinesq equation of water waves
in \eqref{b-e}. We note that the equation \eqref{b-e}
is ill-posed and not all initial conditions lead to solutions. 
Nevertheless, we construct special solutions which are response.

The approach followed in \cite{Rafael13, Rafael17} for similar problems \eqref{b-e}
has two steps. In the first step, one constructed series expansions in $\varepsilon$
that produced approximate solutions. In a second step, one used 
a contraction mapping principle for an operator defined in a
small ball near the approximate solutions obtained in the first 
step.  Of course, this approach requires  a very careful choice of 
the spaces in which the approximate solutions lie 
and  the fixed point problems are formulated.  One important consideration is that the spaces are chosen
such that the operators involved map the spaces into themselves. Since some of
the operators involved are diagonal in Fourier series, it is important that the
norms can be read off from the Fourier coefficients. It will also be convenient that we have
Banach algebras properties and that the nonlinear composition operators can be readily estimated. We have to say that it is the idea in \cite{Rafael13, Rafael17} that inspires our present treatment for the
equations \eqref{000.0} and \eqref{b-e}.

To motivate the procedure adopted in this paper, we note that in the method
of \cite{Rafael13,Rafael17}, the fixed point part does not depend on
any arithmetic condition on the forcing frequency. 
 We will modify slightly the fixed
point part to get response solutions with some regularity for our model \eqref{000.0}.
In this way, we first reformulate the existence of response solutions for equation \eqref{000.0} as a fixed point problem. Then, under certain regularity assumptions for the
nonlinearity and  the forcing, we obtain the response solutions with  corresponding regularity
 on $\varepsilon$ when $\varepsilon$ ranges over
an appropriate domain without any circle centered at the origin $\varepsilon=0$.  It is quite 
possible that the response solutions constructed are not differentiable 
with respect to $\varepsilon$ at $\varepsilon  = 0$ (see Remark~\ref{noo-de}) since we do not assume any Diophantine conditions for the frequency $\omega$. Therefore,  when we consider 
$\varepsilon$ goes to $0$, we just get the response solutions depend continuously on $\varepsilon$.

The method of the proof in  this paper (very different from 
resumming expansions)  consists in transforming the
original equations \eqref{000.0} and \eqref{b-e} into fixed point
equations (see \eqref{fixeq3} and \eqref{plo1}, respectively). The main observation that allows us to solve the fixed point equations is that we are allowed to use the strong
dissipation in the contraction mapping principle.

Our  method also works for 
finitely differentiable problems. 
In such case, we will introduce Sobolev spaces, in which the norms of functions  are measured by size of the Fourier coefficients.

We think that the regularity results obtained in this paper are close to optimal. 
As for the optimality for the domain, 
we find that there exist arbitrarily
small values of $\varepsilon$ for which the map we constructed is not
a contraction and the method of the proof breaks down. Therefore, we conjecture  
that this is optimal and that indeed, regular solutions do not exist for 
these small parameter values and  general forcing and nonlinearity.
We also show in Remark~\ref{noo-de} that, both in the analytic and 
in the finitely differentiable case, there are examples in which the 
solution is not differentiable in $\varepsilon$ at $\varepsilon = 0$ when we remove the Diophantine condition on the forcing frequency $\omega$.

The lack of differentiability at $\varepsilon = 0$ 
is a reflection of the problem being a singular
perturbation. In the case considered here that there are no
non-resonance conditions on the frequency,  the problem is 
more severe than in previously considered cases. 
\subsection{Some possible generalization}
Our method could deal easily with the general case with the form of 
\begin{equation}\label{general}
\begin{split}
\mathbf{p}x_{tt}+\frac{1}{\varepsilon} \mathbf{q}x_t+ g(x,\omega t)= f(\omega t),\quad x\in \mathbb{R}^n,
\end{split}
\end{equation}
where $\mathbf{p},\,\mathbf{q}$ are diagonal constant matrix and $g(x,\omega
t)=Ax+\hat{g}(x,\omega t)$, where $A$ is a matrix in Jordan Block form and $\hat{g}(x,\omega
t)\,:\mathbb{R}^n\times \mathbb{T}^d\rightarrow \mathbb{R}^n$ is
sufficiently regular. We leave the easy details to the interested
readers. See Remark~\ref{in-jordan}, which gives some simplified calculations  after we have carried out the case in \eqref{000.0}.

\subsection{Organization of this paper}
Our paper is organized as follows: In Section~\ref{sec:formulation},
we  present the idea of reformulating the existence of response solutions for equation \eqref{000.0} as a fixed point problem. To solve this fixed point equation, in Section~\ref{sec:spacesode}, we 
 give the precise function spaces that we work in and we
list their important properties, such as Banach algebra properties and
the regularity of the composition operators. We state our three main results: analytic case,
highly differentiable case and low regularity in Section
~\ref{sec:statement}.   Section
~\ref{sec:analytic} is mainly devoted to  the proof of our analytic
result by contraction mapping principle. In the process, we need to
pay more attention to the invertibility of operators and regularity of
composition operators. In Section~\ref{sec:finitely}, we prove our regular
result in  the finitely differentiable case by the contraction argument and the implicit function theorem.
Section~\ref{sec:pde} is an application  to the ill-posed PDE
\eqref{b-e} by a similar idea to used for ODE \eqref{000.0}.

\section{The formulation for  equation \eqref{000.0}}\label{sec:formulation} 
 
In this section, we give an 
overview of our treatment for  ODE model \eqref{000.0}, which can be rewritten as
\begin{equation}\label{000.1}
\begin{split}
\varepsilon x_{tt}+x_t+\varepsilon g(x)=\varepsilon f(\omega t),\,\, x\in \mathbb{R}^n,
\end{split}
\end{equation}
where, as indicated before, the mappings $g:\mathbb{R}^n\rightarrow
\mathbb{R}^n,\,\,f:\mathbb{T}^d\rightarrow \mathbb{R}^n$.  We will reduce the existence of response solutions of equation \eqref{000.1}
to an equivalent fixed point problem. To this end, it is crucial to  make some assumptions for equation \eqref{000.1}.
\subsection{Preliminaries}
For the analytic and highly
differentiable functions $f$ and $g$ defining the equation
\eqref{000.1}, we assume that:
\medskip 

$\mathbf{H}$: The  average of $f$ is $0$ and $g(0)=0$.  Denote $A=Dg(0)$, which  is a $n\times n$ matrix, the spectrum $\lambda_j\,(j=1,\cdots,n)$ of $A$ is real and $\lambda_j\neq 0$.  
\medskip

Actually, we could weaken the assumptions on the regularity of the function
$g$ when considering low regularity results (e.g, $L^2$ or $H^1$). As we will see in Section~\ref{sec:lowregularity}, instead of assuming $g$
is differentiable, we just assume that:
\medskip 

$\mathbf{\widetilde{H}}$:
$g$ is Lipschitz in $\mathbb{R}^n$ and it can be expressed in the form of 
\begin{equation*}
g(x)=Ax+\hat{g}(x),
\end{equation*}
where $A$ is a $n\times n$ matrix and its spectrum is real and
nonzero. Moreover, the nonlinear part $\hat{g}$ satisfies that
$\text{Lip}(\hat{g})\ll 1$ in the whole of $\mathbb{R}^n$. 
\medskip

Note that in both assumptions $\mathbf{H}$ and $\mathbf{\widetilde{H}}$, we are not including that the matrix $A$ is
diagonalizable. Non-diagonalizable matrices appear naturally when
considering oscillators at resonance, which is often a design goal in
several applications in electronics or appear in mechanical systems with several nodes.

We emphasize that the assumption 
$\mathbf{\widetilde{H}}$ involves assumptions on $\hat g$ for all 
values of its argument. This is needed when we consider solutions in 
$L^2(\mathbb{T}^d)$ which may be unbounded.

It is important to note that, once we have established the conclusion
 for $g$ under the assumption $\mathbf{\widetilde{H}}$, we can
 accommodate several physical situations such as piecewise linear
 nonlinearity with small breaks.
\medskip 

Without loss of generality, we  assume that 
\begin{equation}\label{non-res}
\omega \cdot k\neq 0,\,\forall k\in \mathbb{Z}^d\setminus\{0\}.
\end{equation}
Indeed, if there is a  $k_0\in \mathbb{Z}^d\setminus\{0\}$ such that $\omega \cdot k_0=0$, we could reformulate the forcing with only $(d-1)-$dimensional variables which are orthogonal to $k_0$. Namely, the map $f:\,\mathbb{T}^{d-1}\rightarrow \mathbb{R}^n$.

The condition \eqref{non-res}  is called the
$\emph{``non-resonance''}$ condition. If the non-resonance condition
\eqref{non-res} is satisfied, the set $\{\omega t\}_{t \in
  \mathbb{R}}$ is dense on $\mathbb{T}^d$.
\subsection{Quasi-periodic solutions, hull functions}
In this paper, we are interested in finding the quasi-periodic solutions with frequency $\omega\in \mathbb{R}^d$. These are functions of time $t$ with the form
\begin{equation}\label{realso}
x_\varepsilon(t) = U_\varepsilon(\omega t)
\end{equation}
for a suitable function $U_{\varepsilon}:\mathbb{T}^{d} \rightarrow \mathbb{R}^n$, indexed by the small parameter $\varepsilon$. 
The function  $U_{\varepsilon}$
is often called the \emph{``hull function''}. 
Substituting \eqref{realso} into equation \eqref{000.1} and using that
 $\{\omega t\}_{t \in \mathbb{R}}$ is dense in $\mathbb{T}^d$, 
  we obtain that 
\eqref{000.1} holds for a continuous function $U_\varepsilon$ if and only if the hull function
$U_\varepsilon$ satisfies
\begin{equation}\label{fixeq}
\varepsilon\left( \omega\cdot \partial_{\theta}\right) ^2 U_{\varepsilon}(\theta)+\left( \omega\cdot \partial_{\theta}\right) U_{\varepsilon}(\theta) 
+\varepsilon g(U_{\varepsilon}(\theta))=\varepsilon f(\theta).
\end{equation}
Hence, our treatment for equation \eqref{000.1} will be based on finding 
$U_{\varepsilon}$ which solves \eqref{fixeq}. We will manipulate \eqref{fixeq} to reformulate it
as a fixed point problem that can be solved by the contraction argument (or the implicit function theorem).

The equation \eqref{fixeq} we will solve involves parameter $\varepsilon$ (the
inverse of coefficient multiplying the damping). We will obtain solutions with  delicate regularity in $\varepsilon$, which are objects in a space of 
functions.
Precisely, in the analytic case (see Section~\ref{sec:analytic}), we will get a solution 
$U_{\varepsilon}$ of equation \eqref{fixeq} depending analytically on  $\varepsilon$ when $\varepsilon$ ranges on a
complex domain $\Omega$ which does not include the origin $\varepsilon
= 0$ but so that the origin is in the closure of $\Omega$.  In the finitely
differentiable case (see Section~\ref{sec:finitely}), the solution $
U_{\varepsilon}$ is 
differentiable in $\varepsilon$ when $\varepsilon$ is in a real domain $\widetilde{\Omega}$ which does also not 
include zero but includes it in its closure.

However, when we consider the regularity for the solution $U_{\varepsilon}$ of equation \eqref{fixeq} as $\varepsilon$ goes to $0$,  we get that $ U_{\varepsilon}$
is continuous in $\varepsilon$ in the topologies used in
the fixed point problem (see Lemma~\eqref{continuous}). Moreover, we will show that, in the generality
considered in this paper, there are cases in which the solution is
not differentiable at $\varepsilon = 0$ (see Remark~\ref{noo-de}).

Later, we will develop analogous procedures for the PDE model \eqref{b-e} (see
Section~\ref{sec:pde}).  We anticipate that the treatment is inspired
by this section presenting the formulation for ODE. The unknowns will not take values in
$\mathbb{R}^n$, but rather will take values in a Banach space of
functions.  In addition,  the partial differential equation \eqref{b-e} is ill-posed and its nonlinearity is unbound, which make us do
some more drastic rearrangement for its fixed point equation (see \eqref{ffixeq}).

\subsection{Formulation of the fixed point problem}
\label{sec:formulationfixed}
In this part, we just present the formal manipulations. The precise set up will follow, but it is natural to present first the formal manipulations since the rigorous setting is chosen to make them precise. 

Our goal is to transform equation \eqref{fixeq} into an
equivalent fixed point problem.  We rewrite \eqref{fixeq} as
 \begin{equation}\label{fixeq1}
 \begin{split}
 \varepsilon\left( \omega\cdot \partial_{\theta}\right) ^2 U_{\varepsilon}(\theta)+\left( \omega\cdot \partial_{\theta}\right) U_{\varepsilon}(\theta) 
 +\varepsilon A U_{\varepsilon}(\theta)=\varepsilon f(\theta)-\varepsilon\hat{g}(U_{\varepsilon}(\theta)),
 \end{split}
 \end{equation}
where $A$ is a constant matrix and 
 \begin{equation*}
 \hat{g}(x)=g(x)-Ax.
 \end{equation*}

Note that, in both the analytic case and the highly differentiable case , 
we use assumption {\bf{H}}. It is obvious that 
\begin{equation}\label{nont0}
 \hat{g}(0)=0,\,\,D\hat{g}(0)=0.
\end{equation}
Namely,
\begin{equation*}
  \hat{g}(x)=O(x^2),\,\,D\hat{g}(x)=O(x),
 \end{equation*}
where $ O(x)$ denotes the same order as $x$. As a consequence, $D\hat{g}$ is small (in many sense) in a small neighborhood of the origin $x=0$. We could also assume that $D\hat{g}$ is globally small. This is trivial in the sense of complex analyticity by Liouville's theorem. When $g$ is just Lipschitz, we need that $\text{Lip}(\hat{g})$  is globally small as  condition $\mathbf{\widetilde H}$. 
\medskip

Based on equation \eqref{fixeq1} and denoting by $Id$  the $n\times n$
identity matrix, we introduce the linear operator $\mathcal{L}_{\varepsilon}$
as
\begin{equation}\label{ope}
\mathcal{L}_{\varepsilon}=\varepsilon\left( \omega\cdot \partial_{\theta}\right) ^2 Id +\left( \omega\cdot \partial_{\theta}\right)Id  
+\varepsilon A,
\end{equation}
defined on $n-$dimensional periodic functions of $\theta\in \mathbb{T}^d$. Then,
 \eqref{fixeq1} can be rewritten by
\begin{equation}\label{fixeq2}
\mathcal{L}_{\varepsilon}(U_{\varepsilon}(\theta)) =\varepsilon f(\theta)-\varepsilon\hat{g}(U_{\varepsilon}(\theta)).
\end{equation}

As shown in Section~\ref{sec:linverse},  the operator
$\mathcal{L}_{\varepsilon}$ is boundedly invertible in the  space
$H^{\rho,m}$ defined in Section ~\ref{sec:spacesode} when
$\varepsilon$ ranges in a suitable complex domain. This allows the
equation \eqref{fixeq2} to be transformed into a fixed point problem
as
\begin{equation}\label{fixeq3}
U_{\varepsilon}(\theta) =\varepsilon\mathcal{L}_{\varepsilon}^{-1}\left[  f(\theta)-\hat{g}(U_{\varepsilon}(\theta))\right]\equiv \mathcal{T}_{\varepsilon}(U_{\varepsilon})(\theta),
\end{equation}
where we have introduced the operator $\mathcal{T}_{\varepsilon}$.  For a fixed $\varepsilon$, we
can obtain a solution $U_{\varepsilon}$ for equation \eqref{fixeq3} by the contraction
mapping principle. Further, we want to get a solution $U_{\varepsilon}$ possessing optimal regularity in $\varepsilon$. This can be achieved by considering  operator $\mathcal{T}$ above in a function space consisting of functions regular in $\varepsilon$ (see Section~\ref{sec:analyticitysolution} for analytic case and Section~\ref{sec:finisolution} for highly differentiable case). Specially, in the highly differentiable case,  we will use the classic implicit function theorem to get the results with optimal regularity in $\varepsilon$.  For convenience, we now introduce the operator $\mathbf{T}$ involving the arguments $\varepsilon$ and $U$ as the following:
\begin{equation}\label{impli}
\begin{split}
\mathbf{T}(\varepsilon,U):=U-\mathcal{T}(\varepsilon,U).
\end{split}
\end{equation}
 This  makes it clear to obtain the solution
$U=U_{\varepsilon}$, as a function of $\varepsilon$, having the same regularity as $\mathbf{T}$ by the classical implicit function theorem.

Two subtle points appear in this strategy. One is the invertibility of the linear operator $\mathcal{L}_\varepsilon$ and 
the bound of its inverse. 
 Another is the regularity of 
the composition operator $\hat{g}\circ U$ in \eqref{fixeq3}. We also need to study the dependence on the parameter $\varepsilon$ of the solution $U_{\varepsilon}$ satisfying equation \eqref{fixeq3}.

We observe that the linear operator $\mathcal{L}_{\varepsilon}$ is
diagonal in the basis of Fourier functions. This suggests that we use
some variants of Sobolev (or Bergman) spaces which provide analyticity
-- or in the low regularity case $L^2$ or $H^1$. Hence, it will be
useful that the spaces we consider have norms that can be estimated
very easily by estimating the Fourier coefficients.  The estimates of the Fourier coefficients involves the assumptions that the eigenvalues of $A$ are nonzero real number and
that the range of $\varepsilon$ is restricted to a domain accumulating at the origin $\varepsilon=0$. (See Section~\ref{sec:coe} for details).

For the nonlinear estimates, we need that the composition  operator defined by $\hat{g}\circ U$  is
 smooth considered as a mapping acting on the spaces 
we consider.   The regularity of the composition on the left
by a smooth functions acting on  variants of Sobolev spaces have been widely studied \cite{MR74,AZ90, kappe03}.  In Sections~\ref{sec:spacesode}, we 
will present the precise spaces and some properties in these spaces used to 
implement our program.

\subsection{Some heuristic considerations on the smallness conditions required for the present method}\label{sec:small}
Recall the fixed point equation \eqref{fixeq3}, the operator we consider has the structure 
\begin{equation*}
U =\varepsilon\mathcal{L}_{\varepsilon}^{-1} f-\varepsilon\mathcal{L}_{\varepsilon}^{-1}\hat{g}(U)\equiv \mathcal{T}_{\varepsilon}(U).
\end{equation*}
To solve it by iteration, roughly, we need that the map $U\rightarrow \varepsilon\mathcal{L}_{\varepsilon}^{-1}\hat{g}\circ U$ is a contraction in a domain that contains a ball around $\varepsilon\mathcal{L}_{\varepsilon}^{-1}f$. Of course, the notions of contraction and smallness depend on the spaces under consideration. The results of existence are sharper if we consider spaces of more regular functions and the results of local uniqueness are sharper if we consider spaces of less regular functions.  

Both the contraction properties of $\varepsilon\mathcal{L}_{\varepsilon}^{-1}\hat{g}\circ U$ and the smallness properties of $\varepsilon\mathcal{L}_{\varepsilon}^{-1}f$ are formulated  in appropriate norms (which change with the regularity considered). As we will see in Section~\ref{sec:linverse}, the operator $\varepsilon\mathcal{L}_{\varepsilon}^{-1}$ can be bounded in appropriate norms, which allows us to just consider the smallness of $f$ and the properties of the composition $\hat{g}\circ U$. 

To this end, it is clear that we can trade off some of the smallness assumptions in $\hat{g}$ and $f$. If we are willing to make global assumptions of smallness on $\hat{g}$, we do not need any smallness assumption on $f$. If, on the other hand, we assume that $\hat{g}$ is smooth and $\hat{g}(0)=D\hat{g}(0)=0$, we have that $\hat{g}$ is small (in many senses) in a small neighborhood at the origin. From this point of view, it is necessary to impose smallness condition on $f$ in this small neighborhood.

There are some caveats to these arguments:

In the analytic case, assuming that $D\hat{g}$ is small globally (even bounded) in the whole complex space $\mathbb{C}^n$, Liouville's theorem shows that it is constant, namely, $\hat{g}$ is linear. This makes our result true, but it is trivial and we will not state it. Of course, Liouville's theorem is only a concern for analytic results. 

In the low regularity cases (e.g. $L^2$ or $H^1$ when $d\geq 2$), the range of $f$ may be the whole of $\mathbb{R}^n$, hence we need to make global assumptions on smallness in $\hat{g}$. 
In the case of $H^m$ regularity with $m>\frac{d}{2}$, we  prove our results under two types of smallness assumptions (See Section~\ref{sec:proof-fi}).

We also advance that in the case of $H^1$ regularity, the contraction argument we use will be somewhat more sophisticated. (See Section~\ref{sec:lowregularity}).

\section{Function spaces}
\label{sec:spacesode}
\subsection{Choice of spaces}\label{sec:choice}
To implement the fixed point problem outlined in Section~\ref{sec:formulation}, we need to define precisely function spaces with appropriate norms.  The discussion in Section ~\ref{sec:analytic} will make clear, it is very 
convenient that the norms  can be expressed in terms of the Fourier coefficients of functions. 
In such a case, 
the inverse of the linear operator  $\mathcal{L}_{\varepsilon}$ can be easily estimated just by estimating its Fourier coefficients.  We are allowed to use the base in a such way that the Fourier coefficients of the multiplier operator $\mathcal{L}_{\varepsilon}$ have the Jordan standard form. (See Section~\ref{sec:multiplier}).

We also need the spaces to possess other properties allowing us to control the composition $\hat{g}\circ U$ in \eqref{fixeq3} with ease,
such as Banach algebras properties under multiplication 
and the properties of the composition operators. 
To study the analyticity in $\varepsilon$,  we will define spaces of analytic functions of $\varepsilon$ in Section ~\ref{sec:analyticitysolution}.  In this section,
we use the same notations for Banach spaces as in \cite{Llave09, Rafael13,Rafael16}.

For $\rho \geq 0$, we denote 
\begin{equation*}
\mathbb{T}^{d}_{\rho}=\left\lbrace \theta\in \mathbb{C}^d/(2\pi\mathbb{Z})^d\,:\, \mathrm{Re}(\theta_j)\in \mathbb{T},\,\,|\mathrm{Im}(\theta_j)|\leqslant \rho,\,\,j=1,\ldots,d \right\rbrace.
\end{equation*}

We denote the Fourier expansion of a periodic function 
 $f(\theta)$ on  $\mathbb{T}^{d}_{\rho}$ by
\begin{equation*}
\begin{split}
f(\theta)=\sum_{k\in\mathbb{Z}^{d}}\widehat{f}_{k}e^{\mathrm{i} k\cdot\theta},
\end{split}
\end{equation*}
where $k\cdot\theta=k_1 \theta_1+\cdots+k_d \theta_d$ represents the Euclidean product in $\mathbb{C}^d$ and $\widehat{f}_{k}$ are the Fourier coefficients of $f$.
If $f$ is analytic and bounded on $\mathbb{T}^{d}_{\rho}$, then the Fourier coefficients satisfy the Cauchy bounds 
\begin{equation*}
|\widehat{f}_{k}|\leq M e^{-|k|\rho}
\end{equation*}
with $M$ being the maximum of $|f(\theta)|$ on $\mathbb{T}^{d}_{\rho}$ and $|k|=|k_1|+\ldots+|k_d|$.

\begin{definition}\label{space}
For $\rho \geq 0,\,m,\,d,\,n \in  
\mathbb{N}_{+}$, we denote by	$H^{\rho,m}$ the space of analytic functions $U$ in $\mathbb{T}_\rho^d$ with finite norm $:$
	\begin{equation*}
	\begin{aligned}
	H^{\rho,m}:&=H^{\rho,m}(\mathbb{T}^d)\\
	&=\bigg\lbrace U:\,\mathbb{T}_\rho^d\rightarrow \mathbb{C}^n\,\mid\,\|U\|_{\rho,m}^{2}=\sum_{k\in \mathbb{Z}^d}|\widehat{U}_{k}|^{2}e^{2\rho |k|}
	(|k|^{2}+1)^{m}<+\infty\bigg\rbrace.
	\end{aligned}
	\end{equation*}
\end{definition}

It is obvious that the space
$\big(H^{\rho,m},\,\,\|\cdot \|_{\rho,m}\big)$ is a Banach space and indeed a Hilbert space. From the real analytic point of view, we consider the Banach space $H^{\rho,m}$ of the functions that take real values for real arguments.

For $\rho=0$, $H^{m}(\mathbb{T}^d):=H^{0,m}(\mathbb{T}^d)$ is the
standard Sobolev space, we refer to the references
\cite{taylor3,sobolev} for more details.  In this case, when
$m>\frac{d}{2}$, by the Sobolev embedding theorem (see chapter $2$ and
$6$ in \cite{taylor3}), we obtain that 
$H^{m+l}(\mathbb{T}^d)\,(l=1,2,\cdots)$ embeds continuously into $C^l(\mathbb{T}^d)$.

For $\rho>0$,  functions in the space $H^{\rho,m}$ are analytic in the interior of $\mathbb{T}_\rho^d$ and extend to Sobolev functions 
on the boundary of $\mathbb{T}_\rho^d$. 

\begin{remark} 
As a matter of fact, when $\rho>0$ and $m>d$, the space $H^{\rho,m}$ can be identified with a closed space
of the standard Sobolev space $H^{m}(\mathbb{T}_\rho^d)$ consisting
of functions which are complex differentiable. The manifold
$\mathbb{T}^d_\rho$ has $2d$ real dimension so that, when $m>d$,
the standard Sobolev embedding theorem shows that $H^{\rho,m+l}\,(l=1,2,\cdots)$ embeds continuously into $C^l(\mathbb{T}^d_\rho)$. Since the uniform limit of 
complex differentiable functions is also complex differentiable, we
conclude that our space is a closed space of the standard  Sobolev
 space of $\mathbb{T}^d_\rho$ considered as a $2d-$dimensional real manifold. Several variants of this idea appear already in Bergman spaces in \cite{ReedS75,ReedS72}.

We also point out that the set of functions in $H^{\rho,m}$ which take real values for real arguments is 
a closed set  in $H^{\rho,m}$ (this set is also a linear space over the reals).
 Since we will show that  our operators map this set into itself, we get that the fixed point we 
produce will be such that they give real values for real arguments. 
\end{remark}

\subsection{Properties of the chosen spaces $H^{\rho,m}$ above}
We note several  well-known properties of the space $H^{\rho,m}$ defined in the Section~\ref{sec:choice}, which will 
play a crucial role in what follows. 
\begin{lemma}[Interpolation inequalities]\label{interpolation}
For any $0\leq i\leq m, \,0\leq \nu \leq 1$, denote $s=(1-\nu)i+\nu m$,
there exist constants $C_{i,m}$ depending on $i,m$ such that 
\begin{enumerate}
\item -\emph{Sobolev case:} for $f\in H^{m}$, we have that
\begin{equation}\label{Sobolevint} 
\|f\|_{H^s}\leq C_{i,m}\cdot\|f\|_{H^i}^{1-\nu}\cdot\|f\|_{H^m}^{\nu},
\end{equation}
\item -\emph{Analytic case:} for $\rho>0,\,g\in H^{\rho,m},$ we have that
\begin{equation}\label{Sobolevint2} 
\|g\|_{H^{\rho, s}}\leq C_{i,m}\cdot\|g\|_{H^{\rho, i}}^{1-\nu}\cdot\|g\|_{H^{\rho, m}}^{\nu}.
\end{equation}
\end{enumerate}
\end{lemma}
The inequality \eqref{Sobolevint} is the very standard Sobolev interpolation inequality in the literature
\cite{taylor3,Zehnder75}. 
Since, as mentioned before,  the spaces $H^{\rho,m}(\mathbb{T}^d)$ can be considered as a subspace of 
the standard Sobolev space in $\mathbb{T}_\rho^d$, we also have \eqref{Sobolevint2}.

\begin{lemma}[Banach algebra properties]\label{alge}
	We have the following properties in two cases:
	\begin{enumerate}
		\item -\emph{Sobolev case (see  \cite{sobolev,taylor3}):} Let $m>\frac{d}{2}$, there exists a constant $C_{m,d}$ depending only on $m,d$ such that for 
		$u_1,\,u_2\in H^{m}$, the product $u_1\cdot u_2\in H^{m}$ and 
		\begin{equation*}
		\|u_1u_2\|_{H^{m}}\leq C_{m,d}\|u_1\|_{H^{m}}\|u_2\|_{H^{m}}.
		\end{equation*}
		\item -\emph{Analytic case:} For $\rho>0,\,m>d$, there exists a constant $C_{\rho,m,d}$ depending  on $\rho,m,d$ such that for 
		$u_1,\,u_2\in H^{\rho,m}$, the product $u_1\cdot u_2\in H^{\rho,m}$ and 
		\begin{equation*}
		\|u_1u_2\|_{H^{\rho,m}}\leq C_{\rho,m,d}\|u_1\|_{H^{\rho,m}}\|u_2\|_{H^{\rho,m}}.
		\end{equation*}
	\end{enumerate}
	
	In particular,  $H^{\rho,m}$ is a Banach algebra when $\rho,\,m,\,d$ are as above.
\end{lemma}

To analyze  the operator defined in \eqref{fixeq3},  we also need to estimate the
properties of the composition operator $\hat{g}\circ U$.  The following are well known consequence of Gagliardo-Nirenberg inequalities.
\begin{lemma}[Composition properties]\label{gag-nir}
	We have the following properties in two case:
\begin{enumerate}
	
	\item -\emph{Sobolev case (see  \cite{taylor3,Cala10}):} Let $g\in C^m(\mathbb{R}^n,\,\mathbb{R}^n)$ and assume that $g(0)=0$.
	Then, for $u\in H^{m}(\mathbb{T}^d,\,\mathbb{R}^n)\cap L^{\infty}(\mathbb{T}^d,\,\mathbb{R}^n)$, we have 
	\begin{equation*}
	\begin{split}
	\|g(u)\|_{ H^{m}}\leq \mathbf{c}\|u\|_{L^{\infty}}\left(1+\|u\|_{ H^{m}}\right),
	\end{split}
	\end{equation*}
	where $\mathbf{c}:=\mathbf{c}(\eta)=\sup_{|x|\leq \eta,\,\alpha\leq m}|D^{\alpha}g(x)|$. 
	Particularly, when $m>\frac{d}{2}$
	(so that, by the Sobolev embedding theorem $H^m \subset L^{\infty}$), if $g\in C^{m+2},$ then
	\begin{equation}\label{com-sob}
	\begin{split}
	\|g\circ(u+v)-g\circ u-Dg\circ u\cdot v\|_{ H^{m}}\leq C_{m,d}\|u\|_{L^{\infty}}\left(1+\|u\|_{ H^{m}}\right)\|g\|_{C^{m+2}}\|v\|_{ H^{m}}^2,
	\end{split}
	\end{equation}
	\item -\emph{Analytic case:}
	Let $g:B \rightarrow \mathbb{C}^n $ with $B$ being an open ball around the origin in $\mathbb{C}^n$ and assume that $g$
	is analytic  in $B$. 
Then, for $u\in H^{\rho,m}(\mathbb{T}_{\rho}^d,\,\mathbb{C}^n)\cap L^{\infty}(\mathbb{T}_{\rho}^d,\,\mathbb{C}^n)$ with $u(\mathbb{T}_{\rho}^d)\subset B$, we have 
	\begin{equation*}
	\begin{split}
	\|g(u)\|_{ H^{\rho,m}}\leq C_{u}\|u\|_{L^{\infty}(\mathbb{T}_{\rho}^d)}\left(1+\|u\|_{ H^{\rho,m}}\right),
	\end{split}
	\end{equation*}
where $C_{u}$ is a constant depending on the norm of $u$. In the case of $m>d$, we have that 
\begin{equation*}
\begin{split}
\|g\circ(u+v)-g\circ u-Dg\circ u\cdot v\|_{ H^{\rho,m}}\leq C_{\rho,m,d}\|u\|_{L^{\infty}(\mathbb{T}_{\rho}^d)}\left(1+\|u\|_{ H^{\rho,m}}\right)\|v\|_{ H^{\rho,m}}^2.
\end{split}
\end{equation*}
\end{enumerate}
\end{lemma}

The complete proof of Lemma~\ref{gag-nir} 
can be found in Proposition $3.9$ in \cite{taylor3} or 
Proposition $2.20$ in  \cite{kappe03}, Proposition $1$ in \cite{MR74}.
To make the paper self-contained, we just give an sketch of 
the ideas for the inequality  \eqref{com-sob}, but refer the interested readers to the references above.  

Since
\begin{equation*}
\begin{split}
g\circ(u+v)(\theta)-g\circ u(\theta)-Dg\circ u(\theta)\cdot v(\theta)=\int_{0}^1\int_{0}^{t} D^2g\circ(u+tsv)(\theta)\cdot v^2(\theta)dsdt,
\end{split}
\end{equation*}
we get the desired result by the facts that $D^2g\circ(u+tsv)\in H^m$ and its $H^m$ norm is bounded uniformly in $t,s$ and that $H^m$ is a  Banach algebra under multiplication by  Lemma~\ref{alge}. 
The  range of the derivative $D\hat{g}$ is a $n \times n$ matrix, which can 
be identified with $\mathbb{R}^{n^2}$. Note that the dimension of 
the range of $g$ does not play any role in our arguments.  

The proof of Lemma~\ref{gag-nir} is rather elementary in the analytic case.

As a matter of fact, Lemma~\ref{gag-nir} gives not only the composition operator is differentiable but also presents formula for the derivative. It is easy to check that the same argument leads to higher derivatives of the composition operator if we assume more regularity for function $g$. More precisely, we have the following proposition:
\begin{proposition}[Regularity of composition operators]\label{com-regu}
	We have that following two cases:
\begin{enumerate}
	
	\item -\emph{Sobolev case:}
	Let $m>\frac{d}{2}$. Then, the left composition operator 
	\begin{equation*}
	\mathcal{C}_g:\,H^{m}(\mathbb{T}^d,\mathbb{R}^n)\rightarrow H^{m}(\mathbb{T}^d,\mathbb{R}^n)
	\end{equation*}
	defined by 
	\begin{equation*}
	\mathcal{C}_{g}[u](\theta)=g(u(\theta)),
	\end{equation*}
	has the following properties:
	
	If $g\in C^{m+1}(\mathbb{R}^n,\mathbb{R}^n)$, then $\mathcal{C}_g$ is Lipschitz.
	
	If  $g\in C^{m+l+1}(\mathbb{R}^n,\mathbb{R}^n),\, (l=1,2,\cdots)$, then
	$\mathcal{C}_g$ is $C^{l}$.  
	Moreover, the derivative of the operator $\mathcal{C}_{g}$ is  given by 
	\begin{equation*}
	(D\mathcal{C}_{g}[u]v)(\theta)=Dg(u)v(\theta).
	\end{equation*}
\item -\emph{Analytic case:}	Let $\rho>0$. Assume that $m > d$
and $g\,: B \rightarrow \mathbb{C}^n $, where $B$ is an open ball around the origin in $\mathbb{C}^n$,
is analytic  in $B$. 

Let $u_0\in H^{\rho,m}$ be such that $u_0(\mathbb{T}_\rho^d)\subset B$.
Then for all $u$ in a neighborhood $\mathcal{U}$ of $u_0$ in
$H^{\rho,m}$, the operator 
$\mathcal{C}_{g}\,:\,\mathcal{U}\rightarrow
H^{\rho,m}$ 
is analytic. Moreover, for $v\in H^{\rho,m}$, the derivative of the operator $\mathcal{C}_{g}$ is given by
\begin{equation*}
(D\mathcal{C}_{g}[u]v)(\theta)=Dg(u)v(\theta).
\end{equation*}
\end{enumerate}
\end{proposition}
\begin{proof}
	In fact, Lemma~\ref{gag-nir} shows that the operator $\mathcal{C}_g$ is $C^1$ when $g\in C^{m+2}$. For $g\in C^{m+l+1}$, we can proceed by induction. If we have proved the result for 
	$l$ and the formula for the derivative, we obtain the case for 
	$l+1$.  Indeed, if  $g\in C^{m+l+1}$, we have $\mathcal{C}_g$ is $C^l$. Then, for  $g\in C^{m+l+2}$, $Dg\in C^{m+l+1}$, we get $D\mathcal{C}_g$ is $C^l$ by induction. Namely, $\mathcal{C}_g$ is $C^{l+1}$.
	
	In the analytic case, we start by observing that 
	$u(\mathbb{T}_\rho^d)\subset B$ is a compact set by the Sobolev embedding theorem. Hence, it is at a bounded distance from the boundary of 
	$B$.  If the neighborhood of $u$ is sufficiently small, the range of all 
	the functions will also be contained in $B$.  Then, we obtain our result by Lemma~\ref{gag-nir}. We can also refer to  \cite{Rafael17} for more details.  
	\end{proof}

Note that, for the Sobolev case in Proposition~\ref{com-regu}, the regularity of $\mathcal{C}_g$  is not optimal, we refer to \cite{Rs96,AZ90,kappe03} for more results.
		Note also that, for the analytic case in Proposition~\ref{com-regu}, the result is not the most general result.  There are results in the case of regularity
	that the Sobolev embedding theorem does not give continuity. In these cases, we need to take more care of 
	the ranges of the functions. 
	 Since the functions are differentiable in the
	complex sense, we obtain that the composition operator $\mathcal{C}_{g}$ is differentiable
	in the complex sense by the chain rule to obtain the
	derivative. Further,
	to get that the operator $\mathcal{C}_g$ is analytic, we just recall the Cauchy result that  also holds for functions whose arguments range over a complex Banach
	space. See~\cite{Hil57}. 

\section{Statement of the main results}
\label{sec:statement}

In this section, we state several results for the
model \eqref{000.1}.
These results are aimed at 
different regularity of the forcing $f$: analyticity (Theorem~\ref{mainthm}),
 finite (but high enough) number of derivatives (Theorem~\ref{theom-fi})
and low regularity (Theorem~\ref{theom-conti}).

\begin{theorem}\label{mainthm}

Suppose that $f\in H^{\rho,m}(\mathbb{T}^d)$ for some
$\rho>0,\,m>d$ and $g$ is analytic in an open ball around the origin in the space $\mathbb{C}^n$. If the condition $\mathbf{H}$ is satisfied, then, for $\varepsilon  \in
\Omega(\sigma,\mu)$, where
\begin{equation}\label{domain0}
\begin{split}
\Omega:=\Omega(\sigma, \mu)=\left\lbrace \varepsilon\in \mathbb{C}\,:\,\mathrm{Re}(\varepsilon)\geq \mu\,|\mathrm{Im}(\varepsilon)|,\,\,
\sigma\le|\varepsilon|\leq 2\sigma\right\rbrace
\end{split}
\end{equation}
with $\mu>\mu_0$ for $\mu_0>0$ sufficiently large and $\sigma>0$
sufficiently small, there is a unique solution $U_{\varepsilon}\in
H^{\rho,m}(\mathbb{T}^d)$ for equation \eqref{fixeq}.

 Furthermore,
considering $U_{\varepsilon}$ as a function of $\varepsilon$, the mapping $\varepsilon\rightarrow
U_{\varepsilon}:\,\Omega\rightarrow H^{\rho,m}(\mathbb{T}^d)$ is
analytic when
$m>(d+2)$. 

In addition, as $\varepsilon\rightarrow 0$, the solution
$U_{\varepsilon}\rightarrow 0$ and the mapping $\varepsilon\rightarrow
U_{\varepsilon}$ is continuous.
\end{theorem}
\begin{remark}
The statement of Theorem~\ref{mainthm} does not impose any Diophantine
condition on the forcing frequency $\omega$.  Since we do not expand the
solution as a power series in $\varepsilon$, there is no equation
involving the small divisor appearing. We will, however, not get that
the solution is differentiable with respect to $\varepsilon$ at the origin
$\varepsilon = 0$ and this may indeed be false in the generality considered in this paper. (See Remark~\ref{noo-de}).
\end{remark}
\begin{theorem}\label{theom-fi}
Suppose that $f\in H^{m}(\mathbb{T}^d)$ with $m>\frac{d}{2}$ and $g\in
C^{m+l}(\mathbb{R}^n,\mathbb{R}^n) \,\,(l=1,2,\cdots)$.
If the condition $\mathbf{H}$ is satisfied, then,
for $\varepsilon\in \widetilde{\Omega}(\sigma)$, where
\begin{equation}\label{fini-para}
\widetilde{\Omega}:=\widetilde{\Omega}(\sigma)=\left\lbrace \varepsilon\in \mathbb{R}:\,
\sigma\le|\varepsilon|\leq 2\sigma\right\rbrace
	\end{equation}
with sufficiently small $\sigma>0$, there exists a unique
	solution $U_{\varepsilon}\in H^{m}(\mathbb{T}^d)$ for  equation \eqref{fixeq}.

 Moreover,  we have the following regularity in $\varepsilon$: 
 
 	If $g\in C^{m+1}(\mathbb{R}^n,\mathbb{R}^n)$, then the mapping $\varepsilon\rightarrow
 	U_{\varepsilon}:\widetilde{\Omega}\rightarrow H^m(\mathbb{T}^d)$  is Lipschitz.
 
 If  $g\in C^{m+l+1}(\mathbb{R}^n,\mathbb{R}^n)$, then the mapping $\varepsilon\rightarrow
 U_{\varepsilon}:\widetilde{\Omega}\rightarrow H^m(\mathbb{T}^d)$ is
  $C^{l}$.\\
In addition, when $\varepsilon\rightarrow 0$, the solution
 $U_{\varepsilon}\rightarrow 0$ and $\varepsilon\rightarrow
 U_{\varepsilon}$ is continuous.
\end{theorem}
We note that the regularity in $\varepsilon$ in Theorem~\ref{theom-fi} depends on the regularity of the composition operator $g\circ u$ in Proposition~\ref{com-regu}. Even if we show that the derivatives with respect to $\varepsilon$ exist for all $\varepsilon>0$, we do not make any claim about the limit of the derivatives as $\varepsilon$ goes to $0$.
\medskip

The following Theorem~\ref{theom-conti} is for the situation when the forcing and 
the nonlinearity are rather irregular. 

\begin{theorem}\label{theom-conti}
	
Suppose that $f\in L^2(\mathbb{T}^d)$ and
$g$ is globally Lipschitz continuous on $\mathbb{R}^n$ satisfying the condition
$\mathbf{\widetilde{H}}$.  Then, for $\varepsilon\in \widehat{\Omega}\subset\mathbb{R}\setminus\{0\} $
being the sufficiently small domain,
there is a unique solution $U_{\varepsilon}\in L^2(\mathbb{T}^d)$ for
equation \eqref{fixeq}. The solution $U_{\varepsilon}$ is continuous in $\varepsilon$.

Under the above assumptions if $f\in H^1(\mathbb{T}^d)$ and $g\in C^{1+Lip}$, then, the unique
solution $U_{\varepsilon}$ constructed above is in $\cap_{0 \le s < 1} H^s$.
\end{theorem}

Note that Theorem~\ref{theom-conti} applies to some piecewise linear
models (the Lipschitz constant of the derivatives has to be
sufficiently small). Such models appear naturally in many areas. 
 
 We
also stress that in Theorem~\ref{theom-conti}, for $f\in
H^1(\mathbb{T}^d)$, we cannot claim that the solution is in 
$H^1$, but only that it belongs to the intersection $\cap_{0 \leq s < 1}H^s$.  
We do not have a contraction argument in this case, but we can estimate the speed of 
convergence of the iterative procedure in the space $H^s$ for $0\leq s <1 $.

In the analytic case (Theorem~\ref{mainthm}) and in the  highly differentiable regularity (Theorem~\ref{theom-fi}), when $m>(\frac{d}{2}+2)$, we have that
the solution $U_{\varepsilon}$ is $C^2$ with respect to the argument $\theta$. Hence, the quasi-periodic solutions $x(t)$ obtained through \eqref{realso} is also a twice differentiable function of time. As a consequence, the solutions we have produced satisfy the differential equation \eqref{000.0} in the 
classical sense. In the lower regularity case,
the solutions we produce solve the equation in the sense that the
Fourier coefficients of \eqref{fixeq} are the same in both sides. This
is equivalent to solving \eqref{000.0} in the weak sense since the trigonometric polynomials are dense in the space of $C^{\infty}$ test functions. 

In this paper, we also present some results for PDE's model
\eqref{b-e}.  Since the formulation requires new definitions and
auxiliary lemmas, we postpone the formulation of the results till
Section~\ref{sec:pde}.

\section{Analytic case: Proof of Theorem~\ref{mainthm}}
\label{sec:analytic}

We prove  Theorem~\ref{mainthm}
 in the analytic sense by considering the fixed point  equation \eqref{fixeq3} in the Banach space
$H^{\rho,m}$ for any $\varepsilon\in \Omega(\sigma,\mu)$. Recall the equation \eqref{fixeq3}
\begin{equation}\label{ffix}
U_{\varepsilon}(\theta) =\mathcal{L}_{\varepsilon}^{-1}\left[ \varepsilon f(\theta)-\varepsilon\hat{g}(U_{\varepsilon}(\theta))\right]\equiv 
\mathcal{T}_{\varepsilon}(U_{\varepsilon})(\theta).
\end{equation}

The first concern is the invertibility of the linear operator $\mathcal{L}_{\varepsilon}$ and
 the quantitative bounds on its inverse when $\varepsilon$ ranges over the complex domain $\Omega(\sigma, \mu)$ defined in \eqref{domain0}. 
We remark that it is impossible to obtain the same bounds if $\varepsilon$ belongs to the imaginary axis. In fact, we conjecture that the optimal domain of $\varepsilon$, when the solution $U_{\varepsilon}$  of equation \eqref{ffix} is  considered  as a function of  $\varepsilon$, do not extend to the imaginary axis.

Secondly, since we want to obtain a solution $U_{\varepsilon}$ 
analytic in $\varepsilon$, we will define a space consisting of functions analytic in $\varepsilon$. (See the space $H^{\rho,m,\Omega}$
defined in Section~\ref{sec:analyticitysolution}). By reinterpreting
 the fixed point problem in the space $H^{\rho,m, \Omega}$, we obtain  rather directly the analytic dependence on $\varepsilon$ of the
solutions $U_{\varepsilon}$.
The delicate steps are to show that the operator $\mathcal{T}$ defined in \eqref{ffix} maps a ball centered at the origin  
in the 
space $H^{\rho,m, \Omega} $ to itself and it is a contraction in this ball.

\subsection{Estimates on the inverse operator $\mathcal{L}_{\varepsilon}^{-1}$}
\label{sec:linverse}

For the analytic nonlinearity $g$, the linear part $A$ is dominant with
respect to the nonlinear part $\hat{g}$. Moreover, the Lipschitz constant of $\hat{g}$ can be small enough in
a sufficiently small domain.  

We now study the linear operator defined by 
\begin{equation*}
\mathcal{L}_{\varepsilon}=\varepsilon\left( \omega\cdot \partial_{\theta}\right) ^2 Id+\left( \omega\cdot \partial_{\theta}\right) Id
+\varepsilon A.
\end{equation*}
Our main result in this section includes 
 that $\mathcal{L}_{\varepsilon}$  is boundedly invertible from the analytic function space $H^{\rho,m}$ to itself when $\varepsilon$
ranges over a complex conical domain $\Omega(\sigma, \mu)$, which is away from imaginary axis. Of course, 
this result requires the assumptions on $A$ in $\mathbf{H}$. 

A key ingredient for the result is that the norms of the functions can be read off from the sizes of the Fourier series 
and that the operator $\mathcal{L}_\varepsilon$ acts in a very simple matter in Fourier series. 
Indeed, if the matrix $A$ was diagonal, the operator $\mathcal{L}_\varepsilon$ will be just a Fourier multiplier 
in each component (this case is worth keeping in mind as a heuristic guide). 
\subsubsection{Some elementary manipulations}\label{sec:multiplier}

A  consequence of the assumption $\mathbf{H}$ is that there exists a basis of generalized eigenvectors $\Phi_i\in \mathbb{C}^n\,(i=1,2,\cdots,n)$ such that 
\begin{equation}\label{jordan}
\begin{split}
A\Phi=J\Phi,\,\, \Phi=(\Phi_1,\cdots,\Phi_n)^{\top},
\end{split}
\end{equation}
where $J$ is the standard Jordan normal form. That is,
\begin{equation*}
\begin{split}
J=
\left( 
\begin{matrix}
J_1& \,\,& \,\,& 0 \\
\,\,& J_2& \,\,&\,\,\\
\,\,&\,\,&  \ddots&\,\,\\
\,0&\,\,& \,\,& J_p
\end{matrix}
\right) ,\,\,J_j=
\left( 
\begin{matrix}
\lambda_j& \,\,& \,\,0 \\
1& \lambda_j& \,\,\\
\,\,&\ddots&  \ddots\\
 \,\,\,\,\,\,\,0&\,\,& 1& \lambda_j
\end{matrix}
\right) ,\,\,1\leq j\leq p,\,1\leq p \leq n.
\end{split}
\end{equation*}
\medskip 

When we write a function $U_{\varepsilon}(\theta)\in H^{\rho,m}$ in the Fourier expansion as
 \begin{equation*}
U_{\varepsilon}(\theta) =\sum_{k\in\mathbb{Z}^{d}}\widehat{U}_{k,\,\varepsilon}e^{\mathrm{i} k\cdot\theta} =\sum_{k\in\mathbb{Z}^{d} }
\widetilde{\widehat{U}}_{k,\,\varepsilon}\Phi e^{\mathrm{i} k\cdot\theta},      
 \end{equation*}
with $\widehat{U}_{k,\,\varepsilon},\,\widetilde{\widehat{U}}_{k,\,\varepsilon}\in \mathbb{C}^n$ and $\Phi$ being in \eqref{jordan},
 the operator $\mathcal{L}_{\varepsilon}$ acting on the Fourier basis becomes
\begin{equation*}
\begin{split}
\mathcal{L}_{\varepsilon}(\Phi e^{\mathrm{i} k\cdot\theta})=\left( -\varepsilon(k\cdot \omega)^2 Id+\mathrm{i}( k\cdot \omega)Id
+\varepsilon J\right)\Phi e^{\mathrm{i} k\cdot\theta}=:L_\varepsilon(k\cdot \omega)\Phi e^{\mathrm{i} k\cdot\theta},
\end{split}
\end{equation*}
where 
\begin{equation}\label{ma-value}
\begin{split}
L_\varepsilon(a)&=-\varepsilon a^2 Id +\mathrm{i}a Id
+\varepsilon J\\
&=\left( 
\begin{matrix}
L_{\varepsilon,1}(a)& \,\,& \,\,&\,\,0 \\
&L_{\varepsilon,2}(a)& \,\,&\,\,\\
\,\,&\,\,& \ddots&\,\,\\
\,0&\,\,&\,\,& L_{\varepsilon,p}(a)
\end{matrix}
\right)
\end{split}
\end{equation}
with
\begin{equation*}
\begin{split}
L_{\varepsilon,j}(a)=\left( 
\begin{matrix}
l_{\varepsilon,j}(a)& \,\,& \,\,0 \\
\varepsilon& l_{\varepsilon,j}(a)& \,\,\\
\,&\ddots&  \ddots\\
\,0&\,\,& \varepsilon& l_{\varepsilon,j}(a)
\end{matrix}
\right)\,\,(1\leq j\leq p)
\end{split}
\end{equation*}
and
\begin{equation}\label{l-value}
\begin{split}
l_{\varepsilon,j}(a)=-\varepsilon a^2 +\mathrm{i}a+\varepsilon \lambda_j,\,\,
j=1,2,\cdots,p.
\end{split}
\end{equation}

The formula \eqref{ma-value} gives that
\begin{equation}\label{invert0}
\begin{split}
L^{-1}_\varepsilon(a)=\left( 
\begin{matrix}
L^{-1}_{\varepsilon,1}(a)& \,\,& \,\,&\,\,0 \\
&L^{-1}_{\varepsilon,2}(a)& \,\,&\,\,\\
\,\,&\,\,& \ddots&\,\,\\
\,0&\,\,&\,\,& L^{-1}_{\varepsilon,p}(a)
\end{matrix}
\right)
\end{split}
\end{equation}
with 
\begin{equation}\label{invert}
\begin{split}
L_{\varepsilon,j}^{-1}(a)=
\left( 
\begin{matrix}
	l_{\varepsilon,j}^{-1}(a)& \,\,& \,\,&\,\,0&\,\,\\
	-\varepsilon l_{\varepsilon,j}^{-2}(a)& l_{\varepsilon,j}^{-1}(a)& \,\,&\,\,&\\
	\varepsilon^2 l_{\varepsilon,j}^{-3}(a)&\,\,
	-\varepsilon l_{\varepsilon,j}^{-2}(a)&\,\,
	l_{\varepsilon,j}^{-1}(a)&\,\,&\\
	\vdots&\,\,\ddots&\,\, \ddots&\,\, \ddots&\\
	(-1)^{n-1}\varepsilon^{n-1} l_{\varepsilon,j}^{-n}(a)&\,\,\cdots&\varepsilon^2 l_{\varepsilon,j}^{-3}(a)&\,\, -\varepsilon l_{\varepsilon,j}^{-2}(a)\,\,& \,\,l_{\varepsilon,j}^{-1}(a)
\end{matrix}
\right).
\end{split}
\end{equation}

Consequently, to estimate the inverse of $\mathcal{L}_{\varepsilon}$, it suffices to estimate 
\begin{equation}\label{t-value}
\begin{split}
\Gamma_{\varepsilon}:=\sup_{a\in\mathbb{R}}|L^{-1}_\varepsilon(a)|\geq \sup_{k\in\mathbb{Z}^{d}}|L^{-1}_\varepsilon(k\cdot \omega)|.
\end{split}
 \end{equation}
 In the following part, for ease of notation, we will drop the index $j$ in $l_{\varepsilon,j}(a)$ defined in \eqref{l-value}. That means
 $l_{\varepsilon}(a)$ stands for $l_{\varepsilon,j}(a)$. 

\subsubsection{Estimating  the Fourier coefficients $L^{-1}_\varepsilon$ in \eqref{invert0} of the inverse operator $\mathcal{L}^{-1}_{\varepsilon}$}\label{sec:coe}
For the matrix $L_{\varepsilon}(a)$ with special form defined in \eqref{ma-value},  once we obtain the infimum of $|l_\varepsilon(a)|$ in \eqref{l-value} for $a\in
\mathbb{R}$, we get the estimates of
$\Gamma_{\varepsilon}$ defined in \eqref{t-value}. The following estimates are similar to those
in \cite{Rafael13}, which considered only the  $1-$dimensional case. We now
present the details for $n-$dimensional case. 

 Note that  the estimates we obtain also apply to 
the standard Sobolev space $H^m$, which  allows to conclude very quickly the 
results for the finitely differentiable case presented in Section~\ref{sec:finitely}. We  first deal with two special
cases, which throw some light in the general case. 
Of course, from the purely logical point of view, these
special  cases can be omitted since they can be covered in 
the general discussion. We note that $\mathbf{Case~ 1}$ with $\varepsilon\in \mathbb{R}$ is te only case needed in the finite differentiability result. So it is worth dealing with it explicitly.
\medskip

$\mathbf{Case~ 1}$. When $\varepsilon\in \mathbb{R}$,  we have 
\begin{equation*}
\begin{split}
|l_\varepsilon(a)|^{2}&=|-\varepsilon a^2 +\mathrm{i}a
+\varepsilon \lambda_j|^{2}\\
&=(-\varepsilon a^2 +\varepsilon \lambda_j)^2+a^2\\
&=\varepsilon^2 a^4 +(1-2\varepsilon^2 \lambda_j)a^2+\varepsilon^2 \lambda_j^2.
\end{split}
\end{equation*}
Take
$G(v)=\varepsilon^2 v^2 +(1-2\varepsilon^2 )v+\varepsilon^2 \lambda_j^2$ with $v=a^2\geq 0$. It is obvious that $G(v)\geq G(0)
=\varepsilon^2 \lambda_j^2$ since $DG(v)=2\varepsilon^2 v+(1-2\varepsilon^2 \lambda_j)> 0$ due to the smallness of $\varepsilon$. 
Therefore, we have 
\begin{equation}\label{lowbound}
\begin{split}
\inf_{a\in \mathbb{R}}|l_\varepsilon(a)|\geq  |\varepsilon \lambda_j|.
\end{split}
\end{equation}
Namely,
\begin{equation*}
\begin{split}
 \sup_{a\in\mathbb{R}}|l_\varepsilon(a)|^{-1}\leq|\varepsilon \lambda_j|^{-1}.
\end{split}
\end{equation*}
Together with \eqref{invert}, we have that
\begin{equation*}
\Gamma_{\varepsilon}=\sup_{a\in\mathbb{R}}|L^{-1}_\varepsilon(a)|\leq |\varepsilon|^{-1}C_{\lambda}
\end{equation*}
for a positive constant $C_{\lambda}$ depending on the eigenvalues $\lambda_1,\lambda_2,\cdots,\lambda_n$. 

$\mathbf{Case~ 2}$. When $\varepsilon$ is pure imaginary, i.e. $\varepsilon=\mathrm{i}s$ with $
\sigma\leq|s|\leq2\sigma$. 
In this case, there exits a real root $a$ such that $|l_\varepsilon(a)|=0$ since the discriminant $1+4s^2\lambda_j>0$ (by the smallness of $s$) for $-sa^2+a+s\lambda_j=0$. Hence, the operator $\mathcal{L}_{\varepsilon}$ is unbounded if the small parameter $\varepsilon$ locates in the 
imaginary axis, 
which makes the contraction mapping principle inapplicable. 

We conjecture that no solutions for the equation \eqref{fixeq1} exist when $\varepsilon$ is purely imaginary because zero divisors can be considered as resonance. 
\medskip

To study the analyticity in $\varepsilon$ of the function $U_{\varepsilon}$ satisfying \eqref{ffix}, it will be interesting to study the inverse of $\mathcal{L}_{\varepsilon}$ when $\varepsilon$ ranges over the complex domain  $\Omega(\sigma,\mu)$.
\begin{proposition}\label{control1}
For $\Gamma_{\varepsilon}$ defined in \eqref{t-value}, when $\varepsilon \in
\Omega(\sigma,\mu)$, we have 
\begin{equation*}
\Gamma_{\varepsilon}\leq \sigma^{-1}C_{\lambda,\mu}
\end{equation*}
with a  positive constant $C_{\lambda,\mu}$ depending on the eigenvalues $\lambda_1,\lambda_2,\cdots,\lambda_n$ and $\mu$.
\end{proposition}
\begin{proof}
Fix 
\begin{equation*}
\begin{split}
\varepsilon=s_1+\mathrm{i}s_2,
\end{split}
\end{equation*}
for $\varepsilon$ lining on a conical domain $\Omega(\sigma,\,\mu)$, we have $s_1\geq \mu|s_2|$,
where $\mu>\mu_0$ with some sufficiently large positive constant $\mu_0$, and $\sigma ^2\leq s_1^2+s_2^2\leq 4\sigma^2$. Namely,
\begin{equation}\label{small-var}
\begin{split}
\sqrt{1+\frac{1}{\mu^2}}\cdot\sigma\leq s_1\leq \sqrt{1+\frac{1}{\mu^2}}\cdot 2
\sigma.
\end{split}
\end{equation}
Then, one obtains that
\begin{equation}\label{inver-o}
\begin{split}
|l_\varepsilon(a)|^{2}&=|-\varepsilon a^2 +\mathrm{i}a
+\varepsilon \lambda_j|^{2}\\
&=\left[ -s_1 (a^2-\lambda_j) -\mathrm{i}(s_2a^2-a-s_2\lambda_j)\right] ^2\\
&=s_1^2(a^2-\lambda_j)^2 +\left[s_2(a^2-\lambda_j)-a\right]^2.
\end{split}
\end{equation}
If $\lambda_j<0$, it is obvious that
\begin{equation}\label{negative}
	|l_\varepsilon(a)|^{2}\geq s_1^2(a^2-\lambda_j)^2 \geq s_1^2\lambda_j^2.
\end{equation}

The remaining task is to estimate $|l_\varepsilon(a)|^{2}$ in the case of $\lambda_j>0$. Since  $a^2-\lambda_j=0$ holds at the point $a=\pm \sqrt{\lambda_j}$, 
we divide the domain of $a$ into two parts, for $0<\delta\ll 1$, denoted by 
\begin{equation*}
\begin{split}
I_{1}=[(1-\delta)\sqrt{\lambda_j},\,(1+\delta)\sqrt{\lambda_j}]\cup
[(-1-\delta)\sqrt{\lambda_j},\,(-1+\delta)\sqrt{\lambda_j}],\quad
I_{2}=\mathbb{R}\setminus I_{1}.
\end{split}
\end{equation*}
When $a\in I_{2}$, we obtain the estimate
\begin{equation}\label{i2}
\begin{split}
|l_\varepsilon(a)|^{2}\geq s_1^2(a^2-\lambda_j)^2 \geq  s_1^2C_{\lambda},
\end{split}
\end{equation}
where $C_{\lambda}$ depends on the choice of $\delta$ as well.
When $a\in I_{1}$, it is clear that $\left[s_2(a^2-\lambda_j)-a\right]=O(s_2)-a$. Therefore, 
\begin{equation}\label{i1}
\begin{split}
|l_\varepsilon(a)|^{2}&\geq \left[s_2(a^2-\lambda_j)-a\right]^2=\left[O(s_2)-a\right]^2\geq \frac{a^2}{2}
\geq C_{\lambda}\geq C_{\lambda}s_1^2
\end{split}
\end{equation}
by the smallness of $s_1$ and $s_2$. Note that the last inequality in above estimate is very wasteful but we want to get estimates comparable to the ones we have in the other pieces. The inequalities \eqref{small-var}, \eqref{negative}, \eqref{i2} and \eqref{i1} allow that
\begin{equation}\label{sup-l}
	\sup_{a\in\mathbb{R}}|l_{\varepsilon}(a)|^{-1}\leq s_1^{-1} C_{\lambda}\leq \sigma^{-1}C_{\lambda,\mu}.
\end{equation}
Combing with the formulas in \eqref{invert0} and \eqref{invert}, we obtain that
\begin{equation}\label{l-inverse}
\begin{split}
\Gamma_{\varepsilon}=\sup_{a\in\mathbb{R}}|L^{-1}_\varepsilon(a)|\leq \sigma^{n-1}\cdot \sigma^{-n} C_{\lambda,\mu}\leq \sigma^{-1}C_{\lambda,\mu}.
\end{split}
\end{equation}
\end{proof}

It follows from  Proposition~\ref{control1} that, for $\varepsilon\in \Omega(\sigma,\mu)$,
\begin{equation}\label{inverse1}
\begin{split}
|\varepsilon L_{\varepsilon}^{-1}(a)|
\leq \sigma\cdot \sigma^{-1} C_{\lambda,\mu}.
\end{split}
\end{equation}
This inequality is crucial in the contraction mapping argument used in Section~\ref{sec:analyticitysolution}.
\begin{remark}
By \eqref{l-inverse}, we see that
 $\Gamma_{\varepsilon}$ can be bounded by $\sigma^{-1}$ when $\sigma$
 is the minimum distance to the origin in the domain
 $\Omega(\sigma,\mu)$. Then it follows from \eqref{inverse1} that the
 bad factors $\sigma^{-1}$ can be dominated by the good factor
 $\sigma$. This is the reason why we choose $\sigma\leq
 |\varepsilon|\leq 2\sigma$, whose maximum and minimum distance to
 the origin are comparable. Note, however, that the estimate for
 $\varepsilon \mathcal{L}_{\varepsilon}^{-1}$ are independent of
 $\sigma$, so we obtain uniqueness of solutions for different
 $\sigma$, i.e. the solutions obtained for different $\sigma$ agree
for the $\varepsilon$ in the intersection. 
\end{remark}
\begin{remark}\label{in-jordan}
We note that the method presented in this present paper can accommodate small modifications leading to several generalizations. For example, we have the general equation \eqref{general} with $\mathbf{p}=\emph{diag}(\mathbf{p}_1,\cdots,\mathbf{p}_n),\,\mathbf{q}=\emph{diag}(\mathbf{q}_1,\cdots,\mathbf{q}_n)$ being a diagonal matrix satisfying $\mathbf{p}_j,\,\mathbf{q}_j\in \mathbb{R}\setminus \{0\},\,j=1,\,\cdots,n$. In this general case, the only modification with the present exposition is that the calculation for $l_{\varepsilon}(a)$ in \eqref{inver-o} becomes 
\begin{equation*}
\begin{split}
|l_\varepsilon(a)|^{2}&=|-\varepsilon \mathbf{p}_ja^2 +\mathrm{i}\mathbf{q}_ja
+\varepsilon \lambda_j|^{2}\\
&=\left[ -s_1 (\mathbf{p}_ja^2-\lambda_j) -\mathrm{i}(s_2\mathbf{p}_ja^2-\mathbf{q}_ja-s_2\lambda_j)\right] ^2\\
&=s_1^2(\mathbf{p}_ja^2-\lambda_j)^2 +\left[s_2(\mathbf{p}_ja^2-\lambda_j)-\mathbf{q}_ja\right]^2,
\end{split}
\end{equation*}
which makes no difference in our discussion in Proposition~\ref{control1}.
\end{remark}
\subsection{Analyticity in $\varepsilon$ of the solution $U_{\varepsilon}$}
\label{sec:analyticitysolution} 
As the discussion in  
Section~\ref{sec:formulation}, we  
rewrite \eqref{fixeq3} as
\begin{equation}\label{fin-fix}
U(\theta) =\varepsilon\mathcal{L}_{\varepsilon}^{-1}\left[  f(\theta)-\hat{g}(U(\theta))\right]\equiv \mathcal{T}(U)(\theta)
\end{equation}
with $U$ being a function of $\varepsilon$ defined by $U_{\varepsilon}=U_{\varepsilon}(\theta)$. In addition, 
we define the operator $\mathcal{T}$, acting on functions analytic in $\varepsilon$, given  by
\begin{equation}\label{con-op}
\mathcal{T}(U) \equiv\varepsilon\mathcal{L}_{\varepsilon}^{-1}
\left[f-\hat{g}(U)\right]
\end{equation}
with $\mathcal{T}$ being a function of $(\varepsilon,U)$.
Since we want to obtain the solution $U_{\varepsilon}$ depending analytically on 
$\varepsilon$, we reinterpret $\mathcal{T}$ above  as an operator acting on space $H^{\rho,m,\Omega}$ consisting of analytic functions of 
$\varepsilon$ taking values in $H^{\rho,m}$ with $\varepsilon$ ranging over the domain  $\Omega(\sigma, \mu)$. We endow the space 
\begin{equation*}
H^{\rho,m,\Omega}=\bigg\lbrace U:\varepsilon  \rightarrow U_{\varepsilon}:\,\Omega\rightarrow H^{\rho,m} \ \rm{}\ is\ analytic \rm{}\  and \rm{}\  bounded \bigg\rbrace
\end{equation*}
with the
supremum norm
\begin{equation*}
\|U\|_{\rho,m,\Omega}=\sup_{\varepsilon \in \Omega}\|U_{\varepsilon}\|_{\rho,m}.
\end{equation*}
The supremum norm in $\varepsilon$  makes $H^{\rho,m,\Omega}$ a Banach space. Moreover, it is also a
Banach algebra under multiplication when $m>d$ by Proposition~\ref{alge}.
\medskip

We now show that the operator $\mathcal{T}$ defined in \eqref{con-op} maps the space $H^{\rho,m,\Omega}$ into itself.
\begin{lemma}\label{diff1}
	Assume $m>(d+2)$. If $U \in H^{\rho,m,\Omega}$, then  $\mathcal{T}(U)\in H^{\rho,m,\Omega}$. Precisely, if the mapping $\varepsilon \rightarrow U_{\varepsilon}: \Omega\rightarrow H^{\rho,m}$ is
 complex differentiable, then, the mapping $\varepsilon\rightarrow \mathcal{T}_{\varepsilon}
(U_{\varepsilon}): \Omega\rightarrow H^{\rho,m}$ 
is  complex differentiable as well. 
\end{lemma}

\begin{proof}

From the definition \eqref{con-op}, we know that the operator $\mathcal{T}$ is
composed of operators $ \varepsilon \mathcal{L}_{\varepsilon}^{-1}$
and $\hat{g}$. It is clear that the map $\varepsilon \rightarrow
\hat{g}(U_{\varepsilon}):\,\Omega\rightarrow H^{\rho,m}$ is complex
differentiable since  $\hat{g}$ is analytic and it does not depend on $\varepsilon$ explicitly. Therefore, it
suffices to show that the map $\varepsilon\rightarrow \varepsilon
\mathcal{L}_{\varepsilon}^{-1}(V_{\varepsilon}):\,\Omega\rightarrow
H^{\rho,m}$ is complex differentiable when $
V_{\varepsilon}$, considered as a function from $\Omega$ to $H^{\rho,m}$, is a complex
differentiable. 

We prove that the derivatives of
$\varepsilon \mathcal{L}_{\varepsilon}^{-1}(V_{\varepsilon})$ with
respect to $\varepsilon$  exist in the space $H^{\rho,m-2}$
instead of $H^{\rho,m}$. Then, we apply somewhat surprising  Lemma~\ref{bootstrap} in the
Appendix~\ref{sec:appendix}
 to conclude that the derivatives we consider indeed exist  in the space $H^{\rho,m}$.
	 
For a fixed $\varepsilon\in \Omega$, we expand $V_{\varepsilon}(\theta)$ as 
\begin{equation*}
V_{\varepsilon}(\theta) =\sum_{k\in\mathbb{Z}^{d}}\widehat{V}_{k,\,\varepsilon}e^{\mathrm{i} k\cdot\theta},  
\end{equation*}
with 
\begin{equation}\label{coeff0}
\begin{split}
\widehat{V}_{k,\,\varepsilon}=\int_{\mathbb{T}_{\rho}^d} V_{\varepsilon}(\theta)e^{-\mathrm{i} k\cdot\theta}d\theta
\end{split}
\end{equation}
satisfying
\begin{equation}\label{coeff-es}
\begin{split}
\left|\widehat{V}_{k,\,\varepsilon}\right|\leq \left\|V_{\varepsilon}\right\|_{\rho,m} e^{-\rho |k|} \left(|k|^2+1\right)^{-\frac{m}{2}}.
\end{split}
\end{equation}
Taking the derivative with respect to $\varepsilon$ for \eqref{coeff0}, we have that
\begin{equation}\label{coeff1}
\begin{split}
\frac{d}{d\varepsilon}\widehat{V}_{k,\,\varepsilon}=\int_{\mathbb{T}_{\rho}^d} \left(\frac{d}{d\varepsilon}V_{\varepsilon}\right)(\theta)
e^{-\mathrm{i} k\cdot\theta}d\theta
\end{split}
\end{equation}
with 
\begin{equation}\label{coeff-es1}
\begin{split}
\left|\frac{d}{d\varepsilon}\widehat{V}_{k,\,\varepsilon}\right|\leq \left\|\frac{d}{d\varepsilon}V_{\varepsilon}\right\|_{\rho,m}e^{-\rho |k|}(|k|^2+1)^{-\frac{m}{2}}.
\end{split}
\end{equation}
It follows from Section~\ref{sec:linverse} that
\begin{equation*}
\begin{split}
\varepsilon \mathcal{L}_{\varepsilon}^{-1}(V_{\varepsilon})=\sum_{k\in\mathbb{Z}^{d} }
\varepsilon L_{\varepsilon}^{-1}(\omega\cdot k)\widehat{V}_{k,\,\varepsilon}e^{\mathrm{i} k\cdot\theta}
\end{split}
\end{equation*}
with $L_{\varepsilon}^{-1}$ defined in \eqref{invert0}. By \eqref{sup-l}, we have that
\begin{equation*}
\begin{split}
&\left|\frac{d}{d\varepsilon}\left[\varepsilon^n l_{\varepsilon}^{-n}(\omega\cdot k)\right]\right|\\
&=\big|n\cdot \varepsilon^{n-1}\cdot 
l_{\varepsilon}^{-n}(\omega\cdot k)-n\cdot \varepsilon^{n}\cdot 
l_{\varepsilon}^{-n-1}(\omega\cdot k)\cdot \left((\omega\cdot k)^2+\lambda\right)\big|\\
& \leq C_{n,\lambda,\mu}\cdot\sigma^{-1}|k|^2.
\end{split}
\end{equation*}
Together with the formulas \eqref{invert0} and \eqref{invert}, we have that
\begin{equation}\label{de-invert}
\begin{split}
&\left|\frac{d}{d\varepsilon}\left(\varepsilon L_{\varepsilon}^{-1}(\omega\cdot k)\widehat{V}_{k,\,\varepsilon}\right)\right|\\
&\leq\left|\frac{d}{d\varepsilon}\left(\varepsilon L_{\varepsilon}^{-1}(\omega\cdot k)\right)\right|\left|\widehat{V}_{k,\,\varepsilon}\right|+
\left|\varepsilon L_{\varepsilon}^{-1}(\omega\cdot k)\right|\left|\frac{d}{d\varepsilon}\widehat{V}_{k,\,\varepsilon}\right|\\
&\leq C_{n,\lambda,\mu}\cdot \sigma^{-1}|k|^2\left(\left|\widehat{V}_{k,\,\varepsilon}\right|+\left|\frac{d}{d\varepsilon}\widehat{V}_{k,\,\varepsilon}\right|\right).
\end{split}
\end{equation}
Hence, \eqref{coeff-es}, \eqref{coeff-es1} and \eqref{de-invert} yield that
\begin{equation*}
\begin{split}
&\left\|\frac{d}{d\varepsilon}\left(\varepsilon L_{\varepsilon}^{-1}(\omega\cdot k)\widehat{V}_{k,\,\varepsilon}\right)e^{\mathrm{i} k\cdot\theta}
\right\|_{\rho,\,m-\tau}\\&\leq C_{n,\lambda,\mu}\cdot \sigma^{-1}|k|^2\left(\left|\widehat{V}_{k,\,\varepsilon}\right|+\left|\frac{d}{d\varepsilon}\widehat{V}_{k,\,\varepsilon}\right|\right)
\left\|e^{\mathrm{i} k\cdot\theta}\right\|_{\rho,\,m-\tau}\\
&\leq  C_{n,\lambda,\mu}\cdot\sigma^{-1} |k|^2\left(\left\|V_{\varepsilon}\right\|_{\rho,m}+\left\|\frac{d}{d\varepsilon}V_{\varepsilon}\right\|_{\rho,m}
\right)e^{-\rho |k|}(|k|^2+1)^{-\frac{m}{2}}\\
&\ \ \ \ \ \ \ \ \ \ \ \ \ \ \  \ \ \cdot e^{\rho |k|}(|k|^2+1)^{\frac{m-\tau}{2}}\\
& \leq C_{n,\lambda,\mu}\cdot\sigma^{-1}\left(\|V_{\varepsilon}\|_{\rho,m}+\left\|\frac{d}{d\varepsilon}V_{\varepsilon}\right\|_{\rho,m}
\right)(|k|^2+1)^{-(\frac{\tau}{2}-1)}.
\end{split}
\end{equation*}
By $\sum_{|k|=\kappa}1=2^d \kappa^{d-1},\,k=(k_1,\cdots,k_d)\in \mathbb{Z}^d$ and choosing $d+2<\tau\leq m$,  we obtain that
\begin{equation*}
\begin{split}
\sum_{k\in \mathbb{Z}^d}(|k|^2+1)^{-(\frac{\tau}{2}-1)}\leq C_d\sum_{\kappa=0}^{\infty}(\kappa^2+1)^{-\frac{\tau-d-1}{2}}
< \infty.
\end{split}
\end{equation*}
As a consequence, it follows from Weierstrass M-test that the series
\begin{equation*}
\begin{split}
\sum_{k\in \mathbb{Z}^d} \frac{d}{d\varepsilon}\left(\varepsilon L_{\varepsilon}^{-1}
(\omega\cdot k)\widehat{V}_{k,\,\varepsilon}\right)e^{\mathrm{i} k\cdot\theta}
\end{split}
\end{equation*}
converge uniformly on $\varepsilon \in \Omega$ in the space $H^{\rho,m-\tau}$. The fact that these formal derivatives are uniformly convergent shows that they are the true derivatives. Namely,
\begin{equation*}
\begin{split}
\frac{d}{d\varepsilon}\left(\varepsilon \mathcal{L}_{\varepsilon}^{-1}(V_{\varepsilon})\right)=
\sum_{k\in \mathbb{Z}^d} \frac{d}{d\varepsilon}\left(\varepsilon L_{\varepsilon}^{-1}
(\omega\cdot k)\widehat{V}_{k,\,\varepsilon}\right)e^{\mathrm{i} k\cdot\theta}.
\end{split}
\end{equation*}
Therefore, we have that the mapping  $\varepsilon \rightarrow \varepsilon \mathcal{L}_{\varepsilon}^{-1}(V_{\varepsilon}):\Omega\rightarrow H^{\rho,m-\tau}$  is complex differentiable. Since $H^{\rho,m} \subset H^{\rho,m-\tau}$,
 we conclude that the mapping
$\varepsilon \rightarrow \varepsilon \mathcal{L}_{\varepsilon}^{-1}(V_{\varepsilon}):\Omega\rightarrow H^{\rho,m}$  is complex differentiable with derivatives in $H^{\rho,m-\tau}$ by Lemma~\ref{bootstrap} in Appendix~\ref{sec:appendix}.
\end{proof}

\subsection{Existence of the fixed point}
\label{sec:existencefp}
The proof of the existence of the solutions for equation \eqref{fin-fix} is based  on  the fixed point theorem in the Banach space $H^{\rho,m,\Omega}$.  We consider a ball $\mathcal{B}_{r}(0)$ around the origin in $H^{\rho,m,\Omega}$ with radius $r>0$ such that $\mathcal{T}(\mathcal{B}_{r}(0))\subset \mathcal{B}_{r}(0)$  and $\mathcal{T}$ is a contraction in the ball $\mathcal{B}_{r}(0)$.

By \eqref{inverse1}, we get
\begin{equation}\label{cons}
\|\varepsilon \mathcal{L}_{\varepsilon}^{-1}\|_{\rho,m,\Omega}\leq C_{\lambda,\mu}.
\end{equation}
Moreover, it follows from \eqref{nont0} ( $\hat{g}(0)=D\hat{g}(0)=0$)  and  Proposition~\ref{com-regu}
that the Lipschitz constant of the composition operator $\hat{g}\circ U$ is bounded by a constant times the radius $r$ when $U\in \mathcal{B}_{r}(0)$. Therefore, for $U\in \mathcal{B}_{r}(0)$, one has
\begin{equation*}
\begin{split}
\|\mathcal{T}(U)\|_{\rho, m,\Omega}
&\leq\|\mathcal{T}(0)\|_{\rho, m,\Omega}+ \|\mathcal{T}(U)-
\mathcal{T}(0)\|_{\rho, m,\Omega}\\ 
&\leq \|\varepsilon \mathcal{L}_{\varepsilon}^{-1}\|_{\rho,m,\Omega}\left(  \|f\|_{\rho, m,\Omega}+
\| \hat{g}(U)-\hat{g}(0)\|_{\rho, m,\Omega}\right) \\
&\leq C_{\lambda,\mu}\left( \|f\|_{\rho, m,\Omega}+Lip( \hat{g})\cdot \|U\|_{\rho, m,\Omega}\right)
\\
&\leq C_{\lambda,\mu}\left( \|f\|_{\rho, m,\Omega}+O(r)\cdot \|U\|_{\rho, m,\Omega}\right)
\leq r,
\end{split}
\end{equation*}
whenever we take $f$ and $r$ such that
\begin{equation}\label{small-r}
\|f\|_{\rho, m,\Omega}\leq \frac{r}{2C_{\lambda,\mu}},\, C_{\lambda,\mu}O(r)<\frac{1}{2}.
\end{equation}
Note that we need the smallness assumption for $f$ in this case. Thus, $\mathcal{T}(\mathcal{B}_{r}(0)) \subset\mathcal{B}_{r}(0)$. 

For any elements $U_1,\,U_2\in\mathcal{B}_{r}(0)$, we have that
\begin{equation*}
\begin{split}
\|\mathcal{T}(U_1)-\mathcal{T}(U_2)\|_{\rho, m,\Omega}
&=\|\varepsilon\mathcal{L}_{\varepsilon}^{-1} \hat{g}(U_1)-\varepsilon\mathcal{L}_{\varepsilon}^{-1} \hat{g}(U_2)\|_{\rho, m,\Omega}\\
&\leq C_{\lambda,\mu} O(r)\|U_1-U_2\|_{\rho, m,\Omega}\\
&\leq \frac{1}{2}\|U_1-U_2\|_{\rho, m,\Omega}.
\end{split}
\end{equation*} 

Therefore, $\mathcal{T}$ is a contraction in the ball $
\mathcal{B}_{r}(0)$ satisfying \eqref{small-r}. It follows from the fixed point theorem in the Banach space  $H^{\rho,m,\Omega}$ that
there exists a unique solution $U\in H^{\rho,m,\Omega}$ analytic in $\varepsilon$ for equation \eqref{fixeq}.  
\begin{remark}
	When we consider the operator $\mathcal{T}$ defined in \eqref{con-op} in the Banach space 
	$H^{\rho,m,\Omega}$, the solution $U_{\varepsilon}$ obtained
	via fixed point theorem does not lose any regularity on  $\varepsilon$. That is, the solution $U_{\varepsilon}$ naturally depends analytically on the parameter $\varepsilon$.  However, in the finitely differentiable case, when we take $\varepsilon\in \widetilde{\Omega}\subset \mathbb{R}$ instead of $\varepsilon\in\Omega \subset \mathbb{C}$, the contraction mapping principle is not enough to get a solution $U_{\varepsilon}$ with optimal regularity in $\varepsilon$ since when $\rho=0$, the space $H^{\rho,m,\widetilde{\Omega}}$ is no longer a Banach space  with supremum in $\varepsilon$. We will combine with the implicit function theorem to get the optimal regularity. (See Section~\ref{sec:proof-fi} for more details).  It is worth pointing out that in the low regularity, especially in $H^1$, we need more sophisticated contraction argument in some sense since there is no Lipschitz property for the composition operator $\hat{g}\circ u$ in $H^1$. (See Section~\ref{sec:lowregularity}).
\end{remark}
\begin{remark}\label{noo-de}
	We emphasize that  the general solution $U_{\varepsilon}$ obtained above maybe not differentiable in $\varepsilon$ at the origin $\varepsilon=0$ since we do not impose any Diophantine condition for the frequency $\omega$. Indeed, if $U_{\varepsilon}$ was differentiable, we denote the derivative $U^{(1)}(\theta):=\frac{dU_{\varepsilon}(\theta)}{d\varepsilon}\mid _{\varepsilon=0}$ and assume $U_{\varepsilon}=0$ at point $\varepsilon=0$. Then, taking the derivative in $\varepsilon$ at $\varepsilon=0$ for equation~\ref{fixeq}, $U^{(1)}$  would satisfy that 
	\begin{equation}\label{dio}
	\begin{split}
	\left( \omega\cdot \partial_{\theta}\right) U^{(1)}(\theta) 
	=f(\theta).
	\end{split}
	\end{equation}
	If $\omega$ is sufficiently Liouvillean ($e.g,  |\omega\cdot k|\leq \exp
	(-|k|^2)$, such $\omega$ can be easily constructed for infinitely many $k$), we can easily construct analytic function $f$ so that $U^{(1)}(\theta)$ solving \eqref{dio} cannot be even a distribution.
\end{remark}
\begin{lemma}\label{continuous}
For the solution $U_{\varepsilon}$ constructed above, we have that the mapping $\varepsilon\rightarrow U_{\varepsilon}$ is continuous when $\varepsilon \rightarrow 0$.
\end{lemma}
\begin{proof}
We take $\rho_1>\rho>0$ so that both the space  $H^{\rho_1,m}$ and the space $H^{\rho,m}$ satisfy the assumptions of the Theorem~\ref{mainthm}. Denote by $U^1_{\varepsilon},\,U_{\varepsilon}$ the solutions obtained by applying Theorem~\ref{mainthm} to $H^{\rho_1,m}$, $H^{\rho,m}$ respectively. Then, we observe that $U^1_{\varepsilon}=U_{\varepsilon}$ by  $U^1_{\varepsilon}\in  H^{\rho_1,m}\subseteq H^{\rho,m}$
and the uniqueness conclusion in  $H^{\rho,m}$. 
Moreover, we note that the set $\left\lbrace U^1_{\varepsilon}\,|\,\varepsilon\in \overline{\Omega} \right\rbrace $, where $\overline{\Omega}$ denotes the closure of $\Omega$, is bounded in $H^{\rho_1,m}$ and hence it is precompact in $H^{\rho,m}$ topology.

To show that $U_{\varepsilon}$ is continuous in $\varepsilon$ at $\varepsilon=0$, it suffices to verify that the graph $\mathcal{G}$ of $U$. That is,
\begin{equation*}
\mathcal{G}:=\left\lbrace (\varepsilon,\,U_{\varepsilon})|\, \varepsilon\in \overline{\Omega} \right\rbrace 
\end{equation*}
is compact in the $H^{\rho,m}$ topology. Since a ball in $H^{\rho_1,m}$ is precompact in $H^{\rho,m}$, 
we just need to prove that $\mathcal{G}$ is closed. Indeed, the sequence $(\varepsilon_n,\,U_{\varepsilon_n})\in \mathcal{G}$ if and only if  \eqref{fixeq2} is satisfied, that is 
\begin{equation*}
\mathcal{L}_{\varepsilon_n}(U_{\varepsilon_n}(\theta)) =\varepsilon_n f(\theta)-\varepsilon_n\hat{g}(U_{\varepsilon_n}(\theta)).
\end{equation*}
Taking the limits of $\varepsilon_n\rightarrow \varepsilon_{*},\,U_{\varepsilon_n}\rightarrow U^{*}$ for $n\rightarrow \infty$, one can  obtain that
\begin{equation*}
\mathcal{L}_{\varepsilon_*}(U^{*}(\theta)) =\varepsilon_* f(\theta)-\varepsilon_*\hat{g}(U^{*}(\theta)).
\end{equation*}
Hence, we conclude that $(\varepsilon_*,\,U^{*})\in  \mathcal{G}$.
\end{proof}

\section{Finitely differentiable case: Proofs of Theorem~\ref{theom-fi} and Theorem~\ref{theom-conti}}
\label{sec:finitely} 
In this section we present the proof of 
Theorem~\ref{theom-fi}, which concerns
the highly differentiable forcing $f$. We also prove 
Theorem~\ref{theom-conti} 
in which the forcing is assumed to be 
$L^2$ or $H^1$. The method used for the finitely differentiable case, especially $H^1$, is  different  from that for the analytic case.
\subsection{Proof of Theorem~\ref{theom-fi}}\label{sec:proof-fi}
When the forcing term $f$ and the nonlinear term $g$  are finitely differentiable, we consider $\varepsilon \in \mathbb{R}$ in equation \eqref{000.0}. 

\subsubsection{Regularity in $\varepsilon$}\label{sec:finisolution}
In order to get solutions $U_{\varepsilon}$ with some regularity in $\varepsilon$, we need to consider the operator $\mathcal{T}$ defined in \eqref{fixeq3} acting on the space $H^{m,\widetilde{\Omega}}$ 
consisting of  differentiable functions of 
$\varepsilon$ taking values in $H^{m}$ with $\varepsilon$ ranging over the domain  $\widetilde{\Omega}$ defined in \eqref{fini-para}.  Moreover, we endow $H^{m,\widetilde{\Omega}}$ with the
supremum norm
\begin{equation}\label{suprenorm2}
\|U\|_{H^{m,\widetilde{\Omega}}}=\sup_{\varepsilon \in \widetilde{\Omega}}\|U_{\varepsilon}\|_{H^{m}},
\end{equation}
which is similar to the analytic case in Section~\ref{sec:analyticitysolution}. Note that $H^m$ is a Banach space and it is a Banach algebra when
$m>\frac{d}{2}$ by Proposition~\ref{alge}. However,
$H^{m,\widetilde{\Omega}}$ (in contrast with the analytic version $H^{\rho,m,\widetilde{\Omega}}$) is
not a Banach space with the supremum norm defined in \eqref{suprenorm2}.  In this case, if we just apply the
 fixed point theorem to the proof of Theorem~\ref{theom-fi} in the space  $H^{m,\widetilde{\Omega}}$, we may lose some regularity in the argument $\varepsilon$. To avoid this shortcoming, we  
will combine the contraction argument with the implicit function theorem such that the solution $U_{\varepsilon}$ with optimal regularity in $\varepsilon$ can be obtained. 

More precisely, as shown in Section~\ref{sec:fini} at  $\mathbf{Step~1}$,  for some $\varepsilon_0 \in \widetilde{\Omega}$, we first
 produce a solution 
$U_{\varepsilon_0}$ of equation \eqref{fixeq3} such that $\mathbf{T}(\varepsilon_0,U_{\varepsilon_0})=0$, where $\mathbf{T}$ is defined in \eqref{impli}.
 To get the optimal regularity of the map taking  $\widetilde{\Omega}$ to $H^m$,  we  apply the classic implicit function theorem (we refer to the references \cite{dieu, loo, Krantz}) for the operator  $\mathbf{T}$. In this process, it is crucial to
 study the differentiability
of the operator $\mathbf{T}$, 
 mapping
$\widetilde{\Omega}\times H^m
$ to $H^{m}$, with respect to the arguments $(\varepsilon,U)$
as well as the invertibility of 
$D_2\mathbf{T}(\varepsilon_0, U_{\varepsilon_0})$.

By equation \eqref{fixeq3}, we can easily get the differentiability of the operator $\mathcal{T}$ with respect to the argument $U\in H^m$  since the operator $\mathcal{L}_\varepsilon$ are linear 
and the differentiability
properties of the left composition operator  $\hat g \circ U$ are already studied carefully in \cite{kappe03,AZ90}. 

The key to our results will be the differentiability of the operator $\mathcal{T}$ in \eqref{fixeq3} with respect to $\varepsilon$ as the following:
\begin{proposition}\label{derivatives}
	Fix any $m \in \mathbb{N}$ with $m>\frac{d}{2}$ and $\sigma > 0$. 
	We consider the map that
	$\varepsilon \mathcal{L}^{-1}_\varepsilon \in B(H^m, H^m)$ for every $\varepsilon \in \widetilde{\Omega}$, where $B(H^m, H^m)$ denotes the set of bounded operators from the space $H^m$ to itself.
	
	For any $l \in \mathbb{N}$, the map  $\varepsilon \rightarrow 
	\varepsilon \mathcal{L}^{-1}_\varepsilon$ is $C^l$ considered as
	a mapping from $\widetilde{\Omega}$ to $B(H^m, H^m)$.  
	Moreover, for any $l\in \mathbb{N}$ and $\varepsilon \in \widetilde{\Omega}$,
	$\frac{d^l}{d \varepsilon^l} (\varepsilon \mathcal{L}_\varepsilon^{-1})\in B(H^m, H^m)$.
	
	As a matter of fact, something stronger is true. The map 
	$\varepsilon \rightarrow \varepsilon  \mathcal{L}^{-1}_\varepsilon$ is 
	real analytic for $\varepsilon \in \widetilde{\Omega}$ and the radius of
	analyticity can be bounded uniformly for all $\varepsilon \in \widetilde{\Omega}$. 
\end{proposition} 

\begin{proof}  
	The key to the proof  is the observation that, as noted in 
	\eqref{lowbound} in Section~\ref{sec:coe}, $|l_\varepsilon(a)| \ge |\varepsilon| |\lambda_j|\ge \sigma |\lambda_j| $  for $\varepsilon\in \widetilde{\Omega}$.
	
	To study the expansion in powers of  $\delta$  for 
	$l^{-1}_{\varepsilon + \delta}(a) $, we rewrite:
	\begin{equation}\label{goodformu}
	\begin{split} 
	l^{-1}_{\varepsilon + \delta}(a)  &= 
	\left( (\varepsilon +\delta)( \lambda_j -a^2) + \mathrm{i} a  \right)^{-1} \\
	&=\left(\varepsilon ( \lambda_j -a^2) + \mathrm{i} a +
	\delta( \lambda_j -a^2)  \right)^{-1}
	 \\
	&= \left(\varepsilon ( \lambda_j -a^2) + \mathrm{i} a  \right)^{-1} 
	\left( 1 + \delta \frac{ \lambda_j -a^2}{\varepsilon ( \lambda_j -a^2) + \mathrm{i} a}  \right)^{-1}. 
	\end{split} 
	\end{equation}
	It is easy to see that the factor 
	$ \frac{ \lambda_j -a^2}{\varepsilon (
		\lambda_j -a^2) + \mathrm{i} a}$
	is bounded uniformly in $a$ (compute the limit as $|a|$ tends 
	to infinity and observe that the function is continuous in $a$ since 
	the denominator does not vanish)  and uniformly in
	$\varepsilon$ when $\varepsilon$ ranges in an interval bounded away
	from zero.
	
	Therefore, we can expand $\left( 1 + \delta \frac{ \lambda_j -a^2}{\varepsilon ( \lambda_j -a^2) + \mathrm{i} a}  \right)^{-1} $ in \eqref{goodformu} in powers of $\delta$ using  the geometric 
	series formula. Moreover, the radii of convergence are bounded uniformly in $\varepsilon\in \widetilde{\Omega}$ and 
	the values of the coefficients in the expansion are also bounded  uniformly in $a \in \mathbb{R}, \varepsilon \in \widetilde{\Omega}$. 
	
	Using the formula~\eqref{invert} in Section~\ref{sec:multiplier} for the inverse $\mathcal{L}^{-1}_{\varepsilon}$, we also obtain 
	that the matrices $L_{\varepsilon + \delta}^{-1}$ can be  expanded 
	in powers of $\delta$ with coefficients that are  bounded uniformly in $a \in \mathbb{R}, \varepsilon \in \widetilde{\Omega}$.

	We note that the operator $\mathcal{L}^{-1}_{\varepsilon}$
	are multiplier operators (in the sense used in Fourier series). That is, for $\widehat{f}_{k}$ being the Fourier coefficients of function $f$ in the space $H^{m}$, the Fourier coefficients $\widehat{(\mathcal{L}^{-1}_{\varepsilon}f)}_{k}$ of function $(\mathcal{L}^{-1}_{\varepsilon}f)$ in the space $H^{m}$ have the structure:
	\begin{equation}\label{fou-coe}
	\widehat{(\mathcal{L}^{-1}_{\varepsilon}f)}_{k}=L^{-1}_{\varepsilon,k}\widehat{f}_{k},
	\end{equation}
	where each $L^{-1}_{\varepsilon,k}$ is  $n\times n$ matrix (see \eqref{invert} for details). From the discussion in above paragraph, we know that, for each $k$, $L^{-1}_{\varepsilon,k}$ is uniformly analytic in $\varepsilon$.  Thus, we conclude that the operator $\mathcal{L}^{-1}_{\varepsilon}$ is analytic in $\varepsilon$  by \eqref{fou-coe}.

	In addition, we know that the Fourier indices $k$ only 
	enter into the  multipliers $L_{\varepsilon,k}^{-1}$ through 
	$\omega \cdot k$ and  the supremum of $L_{\varepsilon,k}^{-1}$ over the Fourier index is bounded
	by the supremum in $a$, which is studied in the previous Section~\ref{sec:linverse}. 
	Together with the fact that the norms of functions in Sobolev spaces are measured by size of the Fourier coefficients, we have that, for  all $m>\frac{d}{2}$, the norm of $\mathcal{L}^{-1}_{\varepsilon}$ considered as an operator from the Sobolev space $H^m$ to itself is defined by
	\begin{equation}\label{sup}
	\left\|\mathcal{L}^{-1}_{\varepsilon} \right\|_{H^m\rightarrow H^m}=\sup_{k\in \mathbb{Z}^d}\|L^{-1}_{\varepsilon,k}\|.
	\end{equation}
	Note that the norms of $L^{-1}_{\varepsilon,k}$ are just finite-dimensional norms. 
	 	As a consequence,
	we  can bound $\| \mathcal{L}_{\varepsilon }^{-1} \|_{H^m\rightarrow H^m}$
	by the supremum of the multipliers defined in \eqref{sup}. Therefore, when we write $\mathcal{L}^{-1}_{\varepsilon+\delta}=\sum_{n=0}^{\infty}\mathcal{L}^{-1}_{\varepsilon,n}\delta^n$, $\| \mathcal{L}_{\varepsilon,n}^{-1} \|_{H^m\rightarrow H^m}$ can be bounded  by the way of \eqref{sup}. That means $\frac{d^l}{d \varepsilon^l} (\varepsilon \mathcal{L}_\varepsilon^{-1})\in B(H^m, H^m)$ for every $\varepsilon\in \widetilde{\Omega}$.	
\end{proof} 

\subsubsection{Existence of the solutions}\label{sec:fini}
With all the above preliminaries established, now
we turn to finishing the proof of Theorem~\ref{theom-fi}. We divide the proof into two steps. First, for a fixed $\varepsilon\in \widetilde{\Omega}$, we find a fixed point $U_{\varepsilon}$ of $\mathcal{T}$ defined in \eqref{fixeq3}  by considering a domain $\mathcal{P} \subset H^{m}$ with $\mathcal{T}(\mathcal{P})\subset \mathcal{P} $ on which $\mathcal{T}$ is a contraction.  Secondly, we use  the classical implicit function theorem to verify that the solution $U_{\varepsilon}$ we obtained in the first step possesses the optimal regularity in $\varepsilon$.  Namely, we conclude that $U_{\varepsilon} \in H^{m,\widetilde{\Omega}}$.
	\medskip
	
$\mathbf{Step~1}$.  As we state in Section~\ref{sec:small}, there are two ways to prove that $\mathcal{T}$ is a  contraction. One is that we choose a small ball in $H^{m}$  such  that $Lip(\hat{g})$ is small in this ball. Meanwhile, we impose smallness condition on $f$ in this ball. In this way, the operator $\mathcal{T}$ maps this ball into itself and it is a contraction in this ball. (We omit the details here since it is similar to Section~\ref{sec:existencefp}). Another is that we assume that $\emph{Lip}(\hat{g})$ (or $D\hat{g}$) is globally small in the whole of $\mathbb{R}^n$. In this case, for a fixed $\varepsilon\in \widetilde{\Omega}$ and  $U_1,\,U_2\in H^{m}$, it follows from \eqref{cons} that
\begin{equation*}
\begin{split}
\|\mathcal{T}(U_1)-\mathcal{T}(U_2)\|_{H^{m}}
&= \|\varepsilon\mathcal{L}_{\varepsilon}^{-1}(\hat{g}(U_1)-\hat{g}(U_2))\|_{H^{m}}\\
&\leq C_{\lambda,\mu}
Lip(\hat{g})\cdot \|U_1-U_2\|_{H^{m}}\\
&\leq \frac{1}{2}\|U_1-U_2\|_{H^{m}}.
\end{split}
\end{equation*}
This makes $\mathcal{T}$ a contraction in the whole space $H^{m}$. 

In summary, we get a fixed point $U_{\varepsilon_0}\in H^m$ of  the equation \eqref{con-op} for some $\varepsilon_0\in \widetilde{\Omega}$.
\medskip

 $\mathbf{Step ~2}$.
 It follows from  Proposition~\ref{com-regu} and Proposition~\ref{derivatives} that the operator  $\mathcal{T}$ is $C^l$ with respect to the argument $(\varepsilon,U)$. Namely,
$\mathbf{T}(\varepsilon,U):=U-\mathcal{T}(\varepsilon,U)$ is $C^l$
in the domain of $\widetilde{\Omega}\times H^{m}$. Based on $\mathbf{Step ~1}$, we have
$\mathbf{T}(\varepsilon_{0}, U_{\varepsilon_0})=0$. Moreover, $D_2\mathbf{T}(\varepsilon_{0},U_{\varepsilon_0})=Id-D_2\mathcal{T}(\varepsilon_{0},U_{\varepsilon_0})=Id-\varepsilon_0\mathcal{L}_{\varepsilon_0}^{-1}D\hat{g}(U_{\varepsilon_0})$ is invertible since $\varepsilon_0\mathcal{L}_{\varepsilon_0}^{-1}$ is bounded and $D\hat{g}(U_{\varepsilon_0})$ is sufficiently small.
Therefore, by the  implicit function theorem, there exist an open neighborhood included in $\widetilde{\Omega}\times H^{m}$ of $(U_{\varepsilon_0},\varepsilon_0)$ and 
a $C^l$ function $U_{\varepsilon}$ 
satisfying $\mathbf{T}(\varepsilon,U_{\varepsilon})=0$ on this neighborhood.

\subsection{Proof of Theorem~\ref{theom-conti}}
\label{sec:lowregularity}

In this section, we will prove 
Theorem~\ref{theom-conti} in a different way from the first two cases (analytic and highly differentiable cases). 
 The key problem is the properties of the composition operator  $\hat{g}\circ u$
in space  $H^{1}(\mathbb{T}^d)$ or space $L^2(\mathbb{T})$.

\begin{proposition}\label{continuouscomposition}
	For the composition operator 
	defined by:
	\begin{equation}\label{compositiondefined}
	\mathcal{C}_{\hat g}[u](\theta)={\hat g}(u(\theta)),
	\end{equation}
we have the following properties:
	 
If we consider $\mathcal{C}_{\hat{g}}$ acting on $L^2(\mathbb{T}^d,\mathbb{R}^n)$ and assume that $\hat{g}$ is globally Lipschitz continuous on $\mathbb{R}^n$, then
\begin{equation*}
\mathcal{C}_{\hat{g}}:\,L^{2}(\mathbb{T}^d,\mathbb{R}^n)\rightarrow L^{2}(\mathbb{T}^d,\mathbb{R}^n)
\end{equation*}
is Lipschitz continuous.

If we consider $\mathcal{C}_{\hat g}$ acting 
on $H^1(\mathbb{T}^d, \mathbb{R}^n)$ and assume that $\hat{g}\in C^{1+Lip}$, then 
\begin{equation*}
\mathcal{C}_{\hat g}:\,H^{1}(\mathbb{T}^d,\mathbb{R}^n)\rightarrow H^{1}(\mathbb{T}^d,\mathbb{R}^n)
\end{equation*} 
is bounded and  continuous. In particular, given $\epsilon >  0$, 
there is 
$\delta:=\delta(\epsilon, \text{Lip}(\hat g), \hat g(0) ) > 0$ so
that $\|u\|_{H^1} \le \delta $ implies 
$\| \mathcal{C}_{\hat g}(u) \|_{H^1} \le \epsilon$.
\end{proposition}
\begin{proof}
Since $\hat{g}$ is globally Lipschitz continuous on $\mathbb{R}^n$, denote $M=\text{Lip}(\hat g)$ (for ease of notation, we will use $M$ in the following part) and for $u,v\in L^{2}(\mathbb{T}^d,\mathbb{R}^n)$, we get 
\begin{equation*}
|\hat{g}( u(\theta))-\hat{g}(v(\theta))|\leq M |u(\theta)-v(\theta)|.
\end{equation*}
Therefore, 
\begin{equation*}
\|\hat{g}\circ u-\hat{g}\circ v\|_{L^2}\leq M\|u-v\|_{L^2}.
\end{equation*}
We refer to 
\cite{AZ90,Kinder00} for the properties of the operator $\mathcal{C}_{\hat g}$ mapping space $H^1(\mathbb{T}^d,\mathbb{R}^n)$ to itself .
\end{proof}
\begin{remark}\label{small-f}
We emphasize that for our results in $L^2$ and $H^{m}\,\,(m>\frac{d}{2})$ , it is needed to assume that $M=\text{Lip}(\hat g)$ is globally arbitrary small. This allows us to obtain that the operator $\mathcal{T}$ in \eqref{fixeq3} is a contraction in the whole space. 

However, due to the lack of Lipschitz regularity for the operator $\mathcal{C}_{\hat g}$ acting on the space $H^1$ (see \cite{AZ90}), we need to choose a ball belonging to $H^1$ so that the operator $\mathcal{T}$ 
maps this ball into itself. Note that the chosen ball does not need to be small. We also do not require that the forcing is small in $H^1$. 
\end{remark}

Now, we go back to the proof of  Theorem~\ref{theom-conti}.
\begin{proof}

First we give the proof for the result in space $L^2$. By Parseval's identity, we know that the $L^2-$norm is also expressible in terms of the Fourier coefficients. Together with the bound of $\varepsilon\mathcal{L}_{\varepsilon}^{-1}$ in \eqref{cons}, we have that $\mathcal{T}(L^2)\subset L^2$. Also, for $u,v\in L^2$, one has
\begin{equation*}
\|\mathcal{T}(u)-\mathcal{T}(v)\|_{L^2}=\|\varepsilon\mathcal{L}_{\varepsilon}^{-1}\big(\hat{g}\circ u-\hat{g}\circ v\big)\|_{L^2}\leq C_{\lambda,\mu}M\|u-v\|_{L^2}.
\end{equation*}
It follows from $M:=\text{Lip}(\hat g)\ll 1$ in assumption $\mathbf{\widetilde{H}}$ that $\mathcal{T}$ is a contraction in $L^2$. This gives the $L^2$ result.
\medskip

Now, we present the proof for the result in $H^1$. 
Using the interpolation inequality in Lemma~\ref{interpolation}, we obtain, for $n\leq1$, that
\begin{equation}\label{sconvergence} 
\begin{split}
\|\mathcal{T}^{n+1}(u)-\mathcal{T}^n(u)\|_{H^s}
&\leq C_{\lambda,\mu}|\mathcal{T}^{n+1}(u)-\mathcal{T}^n(u)\|_{L^2}^{1-s}\|\mathcal{T}^{n+1}(u)-\mathcal{T}^n(u)\|_{H^1}^{s}\\
&\leq C_{\lambda,\mu} (M^n)^{1-s}\|\mathcal{T}(u)-u\|_{L^2}^{1-s}\|\mathcal{T}^{n+1}(u)-\mathcal{T}^n(u)\|_{H^1}^{s}.
\end{split}
\end{equation}
We have that the function $(M^n)^{1-s}$ is decreasing exponentially.

 The remaining task is to show that $\|\mathcal{T}^{n+1}(u)-\mathcal{T}^n(u)\|_{H^1}$ in \eqref{sconvergence} can be bounded independently of the iteration step $n$.
As a matter of fact, from Proposition~\ref{continuouscomposition}, we know that $u\in H^1$ implies $\hat{g}\circ u\in H^1$. Moreover, it is easy to check that 
\begin{equation*}
\|\hat{g}\circ u\|_{H^1}
\leq  M\|u\|_{H^1}. 
\end{equation*}
Therefore, for the operator $\mathcal{T}$ defined in \eqref{fixeq3}, we get
\begin{equation*}
\|\mathcal{T}(u)\|_{H^1}
=\|\varepsilon\mathcal{L}_{\varepsilon}^{-1}\left(f+\hat{g}\circ u\right)\|_{H^1}\leq  C_{\lambda,\mu}\|f\|_{H^1}+C_{\lambda,\mu}M\|u\|_{H^1}.
\end{equation*}
We now choose a ball $B_r(0)$ centered at the origin in $H^1$ such that  
$B_r(0)$ is mapped by $\mathcal{T}$ into itself. This can be achieved whenever we take $r$ such that $C_{\lambda,\mu}\|f\|_{H^1}+C_{\lambda,\mu}Mr\leq r$, which is equivalent to
\begin{equation}\label{rad}
r\geq \frac{C_{\lambda,\mu}\|f\|_{H^1}}{1-C_{\lambda,\mu}M}. 
\end{equation}
 This can be done since $M$ is small enough. Note that the radius $r$  chosen by \eqref{rad} depends on the function $f$, which can be any function in $H^1$. 
As a consequence, for every $u\in B_r(0)$ and  $n\in \mathbb{N}$, we obtain that $\mathcal{T}^n(u)\in B_r(0)$  and 
\begin{equation*}
\|\mathcal{T}^{n+1}(u)-\mathcal{T}^n(u)\|_{H^1}
\leq  2r. 
\end{equation*}
Thus, \eqref{sconvergence} becomes 
\begin{equation}\label{sconvergence1} 
\begin{split}
\|\mathcal{T}^{n+1}(u)-\mathcal{T}^n(u)\|_{H^s}
\leq C_{\lambda,\mu} (M^n)^{1-s}(2r)^s\|\mathcal{T}(u)-u\|_{L^2}^{1-s},
\end{split}
\end{equation}
which indicates that
the sequence $\mathcal{T}^{n}(u)$ has a limit 
$u^*\in H^s$ and the fixed point obtained by the contraction mapping 
in $L^2$ should be in $H^s$.
Note that  
\eqref{sconvergence1} allows one to bound the distance in $H^s$ from an 
initial guess to the true solution. 
That is,
\begin{equation*}
\begin{split}
\|u^*-u\|_{H^s} &=\|\lim_{n\rightarrow 0}\mathcal{T}^n(u)-u\|_{H^s}\\
&=\|\sum_{n=0}^{\infty}\left[\mathcal{T}^{n+1}(u)-\mathcal{T}^n(u)\right]\|_{H^s}
\\
&\leq C_{\lambda,\mu}(2r)^{s}\|\mathcal{T}(u)-u\|_{L^2}^{1-s} \sum_{n=0}^{\infty}(M^n)^{1-s}\\
&\leq C_{\lambda,\mu}(2r)^{s}(1-M^{1-s})^{-1}\|\mathcal{T}(u)-u\|_{L^2}^{1-s}.
\end{split} 
\end{equation*}
\end{proof}

\begin{remark} 
As shown in \cite{AZ90}, the conditions for composition operators
mapping $H^{1+\delta}$ to itself are very strict. 
There are  many mapping results for the composition in $H^{1 +\delta}
\cap L^\infty$, but it is not clear how the $L^\infty$ norm behaves under
the Fourier multipliers.

Therefore, using the methods of this paper, it seems that there is gap
between the treatments possible for the forcing. Either $H^s\,\,(0<s \leq
1)$ or $H^m\,\,(m > d/2)$.
\end{remark} 

\section{Results for PDEs}
\label{sec:pde} 
 
An important observation is that, since the treatment of \eqref{000.0} 
did not use any properties of the dynamics of equation, we can 
treat even ill-posed  partial differential equations. The ill-posed equation \eqref{b-e} is a showcase 
of the possibilities of our method for model \eqref{000.0}.  The heuristic principle is that we can think of evolutionary  PDE
as models similar to \eqref{000.0} 
in which the role of the phase space $\mathbb{R}^n$ is taken up by 
a function space (of functions of the spatial variable $x$). Note that the nonlinearities in PDE models
can be not just compositions but more complicated operators (even unbounded). For example, the non-linearity $(u^2)_{xx}$ in equation \eqref{b-e} is an unbounded operator from a function space to itself. However, the fixed point problem under consideration in the Banach space we choose overcomes this tricky problem. (See Section~\ref{sec:regularityoperator}).

The solutions produced in this section  point in the direction that ill-posed equations, 
even if, lack a general theory of the existence and uniqueness of the solution, 
may admit many solutions that have a good physical interpretation. 

For convenience,
we rewrite equation \eqref{b-e} as
\begin{equation}\label{b-e1}
\begin{split}
\varepsilon u_{tt}+u_t-\varepsilon \beta u_{xxxx}-\varepsilon u_{xx}=\varepsilon(u^{2})_{xx}+\varepsilon f(\omega t,x),\, x\in\mathbb{T},
\ t\in\mathbb{R},\ \beta>0
\end{split}
\end{equation}
with periodic boundary condition. 

We define the full  Lebesgue measure set
\begin{equation}\label{full}
\begin{split}
\mathcal {O}=\left\{\beta>0: \frac{1}{\sqrt{\beta}}
\textrm{is not an integer}\right\}.
\end{split}
\end{equation}
Note that we shall only work with values of $\beta$ in $\mathcal {O}$ so that the eigenvalues of the linear operator $\varepsilon \beta \partial_{xxxx}+\varepsilon \partial_{xx}$ in \eqref{b-e1} are different from zero in a such way that  the operator $\mathcal{N}_{\varepsilon}$ defined in \eqref{plo} is invertible. (See Section~\ref{sec:regularityoperator} for the details).

\begin{remark} There are other models of friction besides the $u_t$ term in \eqref{b-e1} that one could 
	consider. 
	The treatment given in the present paper is a very general method 
	and  could be applied to several friction models, such as $u_{txx}$.  
	
	We note also that our method for the ill-posed equation \eqref{b-e1} with positive parameter $\beta$ also applies to  well-posed equation \eqref{b-e1} with negative parameter $\beta$. It is even easier for well-posed case since the eigenvalues of the linear operator $\varepsilon \beta \partial_{xxxx}+\varepsilon \partial_{xx}$ in \eqref{b-e1} are not zero such that we can invert the operator $\mathcal{N}_{\varepsilon}$ defined in \eqref{plo}. 
	
	However, we just consider the ill-posed
	model \eqref{b-e1} that serves as motivation for the readers. This ill posed case is what appears in water wave theory \cite{Bou72}.
\end{remark}

\subsection{Formulation of the fixed point problem}
Similar to  Section~\ref{sec:formulation} for ODE model,
we need to reduce the equation \eqref{b-e1} to a fixed point problem. In this section, we just
present the formal manipulations omitting specification of
spaces. Indeed, the precise spaces defined in
Section~\ref{sec:spaces} will be motivated by the desire to justify
the formal manipulations and that the operators considered are a
contraction.

Our goal is to find response solutions of the form 
\begin{equation}\label{realso1}
u_{\varepsilon}(t,x)=U_{\varepsilon}(\omega t,x),
\end{equation}
where, for each fixed $\varepsilon$, $U_{\varepsilon}: \mathbb{T}^d
\times \mathbb{T}\rightarrow \mathbb{R}$. Inserting \eqref{realso1} into \eqref{b-e1}, we get the following functional equation for $U_{\varepsilon}:$
\begin{equation}\label{ffixeq}
\begin{split}
\varepsilon\left( \omega\cdot \partial_{\theta}\right) ^2 U_{\varepsilon}(\theta,x)+\left( \omega\cdot \partial_{\theta}\right) U_{\varepsilon}(\theta,x) 
-\varepsilon\beta\partial_{x}^4 U_{\varepsilon}(\theta,x)&-\varepsilon\partial_{x}^2 U_{\varepsilon}(\theta,x) \\
&=\varepsilon (U_{\varepsilon}^{2})_{xx}+\varepsilon f(\theta,x).
\end{split}
\end{equation}
The solution of equation \eqref{ffixeq} will be the centerpiece of our treatment. 

Denote by $\mathcal{N}_{\varepsilon}$ the linear operator 
\begin{equation}\label{plo}
\mathcal{N}_{\varepsilon}U_{\varepsilon}(\theta,x)=\left[ 
\varepsilon\left(\omega\cdot \partial_{\theta}\right) ^2  +\left( \omega\cdot \partial_{\theta}\right)-\varepsilon\beta\partial_{x}^4 -\varepsilon\partial_{x}^2\right]U_{\varepsilon}(\theta,x).
\end{equation}
Then, \eqref{ffixeq} can be rewritten as
\begin{equation}\label{ffixeq1}
\mathcal{N}_{\varepsilon}U_{\varepsilon}(\theta,x)=\varepsilon (U_{\varepsilon}^{2})_{xx}+\varepsilon f(\theta,x).
\end{equation}
As we will see in Section~\ref{sec:regularityoperator},
the operator $\mathcal{N}_{\varepsilon}$ 
is boundedly invertible in some appropriate space for $\varepsilon \in \Omega(\sigma,\mu)$ defined in \eqref{domain0}. Namely,  \eqref{ffixeq1} becomes
\begin{equation}\label{plo1}
U_{\varepsilon}(\theta,x)=\varepsilon\mathcal{N}_{\varepsilon}^{-1}\left[ (U_{\varepsilon}^{2})_{xx}+f(\theta,x)\right] \equiv \mathcal{T}_{\varepsilon}(U_{\varepsilon}(\theta,x)),
\end{equation}
where, for convenience, we  introduce the operator $\mathcal{T}_{\varepsilon}$. In Section~\ref{sec:mainpr} dealing with the analytic case, we will show that there exists a solution $U_{\varepsilon}$ analytic in $\varepsilon$ for equation
\eqref{plo1} by the contraction mapping argument.  Moreover, in Section~\ref{sec:pde-finitely} carrying out finitely differentiable case, we will combine contraction mapping principle with the classical  implicit function theorem to get the regular results.

From the formal manipulation above,  we find that the first key point is to study the invertibility
of the operator $ \mathcal{N}_{\varepsilon}$ and give quantitative
estimates on its inverse for $\varepsilon$ in a complex domain. Note
that the linear operator $ \mathcal{N}_{\varepsilon}$ defined in \eqref{plo} used to study
 PDE models is much
more complicated than the linear operator 
$\mathcal{L}_{\varepsilon}$ defined in \eqref{ope} for ODE models since  $\mathcal{N}_{\varepsilon}$ involves not only the angle variable $\theta\in \mathbb{T}^d$ but also the space variable $x\in \mathbb{T}$. This leads to 
different calculation for the inverse of $\mathcal{N}_{\varepsilon}$ (See Section~\ref{sec:regularityoperator}). 

The second crucial part is that the nonlinearity $(U^2_{\varepsilon})_{xx}$ maybe  unbounded from one space to itself. However, it happens that $\varepsilon\mathcal{N}_{\varepsilon}^{-1}(U^2)_{xx}$ is bounded. (See Lemma~\ref{derivative-non} and Lemma~\ref{ppo} for more details).

To get a fixed point for equation \eqref{plo1}, analogous to the smallness  arguments in Section~\ref{sec:small} for ordinary partial differential equation \eqref{000.0}, we also need to impose some smallness conditions for partial differential model. However, we only consider a specially nonlinear map $U\rightarrow \varepsilon\mathcal{N}_{\varepsilon}^{-1}(U^2)_{xx}$, which is analytic, be a contraction in a domain that contains a ball around $\varepsilon\mathcal{N}_{\varepsilon}^{-1}f$.  It is nontrivial to choose a sufficiently small ball and  the forcing $f$ is assumed to be small in this ball. 
\subsection{Choice of spaces and the statement of our results}
\label{sec:spaces}

In this section, we give the concrete spaces we work
in. Again, we note that the main principle is that the norms of the functions
needed to be expressed in terms of the Fourier coefficients associated
to the Fourier basis in arguments $\theta$ and $x$.  This
permits us to estimate the inverse of the linear operator
$\mathcal{N}_{\varepsilon}$ just by estimating its Fourier
coefficients.  We also need these spaces to possess the Banach algebra properties and the properties of 
composition operator so that the nonlinear terms can be controlled.
From the point of view analyticity in $\varepsilon$ , it is necessary
to define spaces consisting of analytic functions with respect to $\varepsilon$.

In a way analogous to the definition in Section~\ref{sec:spacesode},
for $\rho\geq 0,\,m,d\in \mathbb{Z}_{+}$, we define the space of
analytic functions $U$ in $\mathbb{T}_\rho^{d+1}$ with finite norm
	\begin{equation*}
	\begin{aligned}
	\mathcal{H}^{\rho,m}:&=\mathcal{H}^{\rho,m}(\mathbb{T}^{d+1})\\
	&=\bigg\lbrace U\,:\,\mathbb{T}_\rho^{d+1}\rightarrow \mathbb{C}\,\mid\,
	U(\theta,x)=\sum_{k\in \mathbb{Z}^d,\,j\in\mathbb{Z}}\widehat{U}_{k,j}e^{\mathrm{i} (k\cdot\theta+j\cdot x)},\\
	&\ \ \ \ \ \ \ \ \ \ \ \|U\|_{\rho,m}^{2}=\sum_{k\in \mathbb{Z}^d,\,j\in\mathbb{Z}}|\widehat{U}_{k,j}|^{2}e^{2\rho (|k|+|j|)}
	(|k|^{2}+|j|^2+1)^{m}<+\infty\bigg\rbrace.
	\end{aligned}
	\end{equation*}
It is obvious that the space
$\big(\mathcal{H}^{\rho,m},\,\,\|\cdot \|_{\rho,m}\big)$ is a Banach space as well as a Hilbert space.

We actually 
consider $\mathcal{H}_0^{\rho,m}$, which is a subspace of $\mathcal{H}^{\rho,m}$, consisting of functions $U\in\mathcal{H}^{\rho,m}$ with 
\begin{equation}\label{0-ave}
	\int_{0}^{2\pi} U(\theta, x)dx=0.
\end{equation}
In the physical applications, we also consider the closed subspace
of $\mathcal{H}^{\rho,m}$ in which the functions take real values for real arguments. 
 
Note that the choice of the normalization condition \eqref{0-ave}
is motivated by the assumption that
\begin{equation*}
\int_{0}^{2\pi} f(\theta, x)dx=0.
\end{equation*} 
Here and after, we consider our fixed point problems in the space 
$\mathcal{H}_0^{\rho,m}$.
To simplicity the notation, we still write $\mathcal{H}^{\rho,m}$  as $\mathcal{H}_0^{\rho,m}$.

For $\rho>0$, $\mathcal{H}^{\rho,m}$ consists of function which are analytic in the domain $\mathbb{T}_{\rho}^{d+1}$. For $\rho=0$, $\mathcal{H}^{m}:=\mathcal{H}^{0,m}$ is just the regular Sobolev space.

Similar to Proposition~\ref{alge}, when $\rho>0,\,m>(d+1)$ or $\rho=0,\,m>\frac{d+1}{2}$, we still have the Banach algebra properties in space $\mathcal{H}^{\rho,m}$.

Now we are ready to state our main results on the  existence of  quasi-periodic solutions for  PDE \eqref{b-e1} in the cases of analyticity and finite differentiability.
\begin{theorem}\label{main-p}
Assume that  $f \in \mathcal{H}^{\rho,m}(\mathbb{T}^{d+1})$ with 
	$\rho>0,\,\,m>(d+1)$. Then, for  $\varepsilon\in \Omega(\sigma,\mu)$ defined in \eqref{domain0}, there exists a unique solution $U_{\varepsilon}\in \mathcal{H}^{\rho,m}(\mathbb{T}^{d+1})$ for equation \eqref{ffixeq}. 
	
Furthermore,
considering $U_{\varepsilon}$ as a function of $\varepsilon$, we have that $\varepsilon\rightarrow
U_{\varepsilon}:\,\Omega\rightarrow \mathcal{H}^{\rho,m}$ is
analytic when
$m>(d+5)$. 	
	In addition, when $\varepsilon\rightarrow 0$, the solution $U_{\varepsilon}\rightarrow 0$ and $\varepsilon\rightarrow U_{\varepsilon}$ is continuous.
\end{theorem}

Our method also applies to  finitely differentiable forcing., but we leave the details.
\begin{theorem}\label{pde-fi}
Assume that
  $f\in \mathcal{H}^{m}(\mathbb{T}^{d+1})$ with $m>\frac{d+1}{2}$.
Then, for 
$\varepsilon\in \widetilde{\Omega}$ defined in
\eqref{fini-para}, there exists a unique solution
$U_{\varepsilon}\in \mathcal{H}^{m}(\mathbb{T}^{d+1})$ for
equation \eqref{ffixeq}.

Furthermore, for any $l\in \mathbb{N}$, the map
 $\varepsilon\rightarrow
U_{\varepsilon}$ is 
 $C^{l}$ (even real analytic) considered as a mapping from $\widetilde{\Omega}$ to $\mathcal{H}^{m}$.
 In addition, when $\varepsilon\rightarrow 0$, the solution $U_{\varepsilon}\rightarrow 0$ and the map $\varepsilon\rightarrow U_{\varepsilon}$ is continuous.
\end{theorem}

\subsection{The boundness of the operator $\mathcal{T}_{\varepsilon}$ defined in \eqref{plo1} taking $\mathcal{H}^{\rho,m}$ into itself}
\label{sec:regularityoperator}
For the PDE model \eqref{b-e1}, the nonlinear map $U\rightarrow (U^2)_{xx}$ (which in the ODE case was a composition operator with $\hat{g}\circ U$) is an unbounded operator from a space to itself. We will show, however, that the map $U\rightarrow \varepsilon\mathcal{N}_{\varepsilon}^{-1}(U^2)_{xx}$ is bounded from a space to itself. To this end, we give the
following lemmas and propositions.  Some of the results would generalize for a nonlinearity of the form $U\rightarrow (g(U))_{xx}$. We will not pursue these specialized
results in this paper, but we think it would be an interesting subject. 
\begin{lemma}\label{derivative-non}
Let $U\in \mathcal{H}^{\rho,m}$. Denote by 
\begin{equation}\label{non-term}
h(U)=(U^2)_{xx}.
\end{equation}
Then, $h$ is analytic from the space  $\mathcal{H}^{\rho,m}$
to the space $\mathcal{H}^{\rho,m-2}$. Moreover, for $V\in \mathcal{H}^{\rho,m}$, we have that
	\begin{equation*}
	\|Dh(U)V\|_{\rho,m-2}\leq 2\|U\|_{\rho,m}\|V\|_{\rho,m}.
	\end{equation*}
\end{lemma}

\begin{proof}
We rewrite $h=h_1\circ h_2$ with
\begin{equation*}	\begin{split}
h_1:\,&\mathcal{H}^{\rho,m}\rightarrow\mathcal{H}^{\rho,m-2}\\
&\ \ \ U\rightarrow U_{xx}
\end{split}
\end{equation*}
and 
	\begin{equation*}
	\begin{split}
	h_2:\,&\mathcal{H}^{\rho,m}\rightarrow\mathcal{H}^{\rho,m}\\
	&\ \ \ U\rightarrow U^{2}.
	\end{split}
	\end{equation*}
It is obvious that both  $h_1$ and $h_2$ are analytic. Therefore, the composition operator $h:\mathcal{H}^{\rho,m}\rightarrow\mathcal{H}^{\rho,m-2}$ is analytic. 
Moreover,
	\begin{equation*}
	Dh(U)V=\frac{d}{d\xi}h(U+\xi V)\bigg\lvert_{\xi=0}=\frac{d}{d\xi}\left((U+\xi V)^2\right)_{xx}\bigg\lvert_{\xi=0}=2(U V)_{xx},
	\end{equation*}
	which shows that 
	\begin{equation*}
	\|Dh(U)V\|_{\rho,m-2}\leq 2\|UV\|_{\rho,m}\leq2\|U\|_{\rho,m}\|V\|_{\rho,m}
	\end{equation*}
	by the Banach algebra property in the space $\mathcal{H}^{\rho,m}$.
\end{proof}
Lemma~\ref{derivative-non} allows that the map $U\rightarrow (U^2)_{xx}$ is bounded from the space $\mathcal{H}^{\rho,m}$  to $\mathcal{H}^{\rho,m-2}$. To prove the boundedness of the operator $\mathcal{T}_{\varepsilon}$ defined in \eqref{plo1}, the remaining task is to show that $\varepsilon\mathcal{N}_{\varepsilon}^{-1}: \mathcal{H}^{\rho,m-2}\rightarrow \mathcal{H}^{\rho,m}$ is bounded. 
\begin{lemma}\label{ppo}
For a fixed $\varepsilon \in \Omega(\sigma,\mu)$, the operator  $\varepsilon\mathcal{N}_{\varepsilon}^{-1}$  taking the space $\mathcal{H}^{\rho,m-2}$ into $\mathcal{H}^{\rho,m}$ is bounded.
\end{lemma}

\begin{proof}
 We verify that $\|\varepsilon\mathcal{N}_{\varepsilon}^{-1}\|_{\mathcal{H}^{\rho,m-2}\rightarrow \mathcal{H}^{\rho,m}}$ can be bounded
 by the supremum of its multipliers, as we argued in the proof of Proposition~\ref{derivatives}.
 
For $V\in\mathcal{H}^{\rho,m-2}$, by \eqref{plo} and \eqref{0-ave}, we have the following Fourier expansion  
\begin{equation*}
 \begin{split}
  \mathcal{N}^{-1}_{\varepsilon}(V(\theta,x))=\sum_{k\in \mathbb{Z}^d \atop j\in\mathbb{Z}
  	\setminus \{0\}}\frac{1}{-\varepsilon(k\cdot \omega)^2+\mathrm{i} (k\cdot \omega)
  	-\varepsilon(\beta j^4-j^2)}\widehat{V}_{k,j}e^{\mathrm{i} (k\cdot\theta+j\cdot x)}.
 \end{split}
\end{equation*}
Note that, by  \eqref{full}, $\beta j^4-j^2\neq 0$ for $j\in\mathbb{Z}
\setminus \{0\}$.

To obtain the desired results, we now estimate the supremum of $\widetilde{\mathbf{N}}_{\varepsilon}$  defined by
\begin{equation}\label{t-es}
\begin{split}
 &\widetilde{\mathbf{N}}_{\varepsilon}(k,j)\\
 &:=
\frac{k^2+j^2}{-\varepsilon (k\cdot \omega)^2+\mathrm{i}(k\cdot \omega)-\varepsilon(\beta j^4-j^2)}\\
&=\left(
\frac{k^2}{-\varepsilon (k\cdot \omega)^2+\mathrm{i}(k\cdot \omega)-\varepsilon(\beta j^4-j^2)}+
\frac{j^2}{-\varepsilon (k\cdot \omega)^2+\mathrm{i}(k\cdot \omega)-\varepsilon(\beta j^4-j^2)}\right)
\end{split}
\end{equation}
for $k\in \mathbb{Z}^d,\,j\in\mathbb{Z}\setminus \{0\}$.
In fact, \eqref{t-es} includes two terms, which have similar estimates, we just give the details for the second term.  Note that it is an easy case for $k=0$. 
We will estimate the infimum of 
\begin{equation*}
\begin{split}
N_{\varepsilon}(a,t):=\frac{-\varepsilon a^2+\mathrm{i}a-\varepsilon(\beta t^2-t)}{t}, \,\,a:=(k\cdot \omega)\in \mathbb{R}\setminus\{0\},\,t:=j^2\in \mathbb{Z}_{+}.
\end{split}
\end{equation*}

Taking $\varepsilon=s_1+\mathrm{i}s_2\in \Omega(\sigma,\mu)$,  we have 
\begin{equation}\label{simi-cal}
\begin{split}
 |N_{\varepsilon}(a,t)|^2=s_1^2\left[\frac{a^2}{t}-(1-\beta t)\right]^2+
\left[s_2\left(\frac{a^2}{t}-(1-\beta t)\right)-\frac{a}{t}\right]^2,
\end{split}
\end{equation}
which has an infimum controlled by $\sigma$  by a similar argument to  Proposition~\ref{control1}. We now estimate \eqref{simi-cal}. When $\beta>1$, we have that $1-\beta t<0$. Thus, $ |N_{\varepsilon}(a,t)|^2\geq s_1^2\left[\frac{a^2}{t}-(1-\beta t)\right]^2\geq (\beta -1)^2s_1^2:=s_1^2C_{\beta}$ for a positive constant $C_{\beta
}$ depending on $\beta$. In the following part, to simplify the notation, we denote $C_{\beta}$ by all  constants depending on $\beta$.

We  focus mainly on the case of $0<\beta<1$. We divide $t\in \mathbb{Z}_{+}$ into two regions as the following:
\medskip

$\mathbf{Case~1}.$
When $t\geq[\frac{1}{\beta}]+1$, we have that $1-\beta t<0$. Therefore 
\begin{equation*}
 |N_{\varepsilon}(a,t)|^2\geq s_1^2\left[\frac{a^2}{t}-(1-\beta t)\right]^2\geq s_1^2C_{\beta
 }.
\end{equation*}

$\mathbf{Case~2}$. When $1\leq t\leq [\frac{1}{\beta}]$, we get that  $t(1-\beta t)\in [C_{\beta
}^{1},\,C_{\beta
}^{2}]$ with $C_{\beta
}^{2}\geq C_{\beta
}^{1}>0$. It is clear that 
$\frac{a^2}{t}-(1-\beta t)=0$ holds at $a^2=t(1-\beta t)\in [C_{\beta
}^{1},\,C_{\beta
}^{2}]$, namely, $a\in[-\sqrt{C_{\beta
}^{2}},-\sqrt{C_{\beta
}^{1}}]\cup[\sqrt{C_{\beta
}^{1}},\sqrt{C_{\beta
}^{2}}]$. 
Now, we define two regions in $a\in \mathbb{R}$, by choosing a constant $0<\delta\ll 1$, as follows
\begin{equation*}
 \begin{split}
  I_1=[(-1-\delta)\sqrt{C_{\beta
  	}^{2}},(-1+\delta)\sqrt{C_{\beta
  	}^{1}}]\cup [(1-\delta)\sqrt{C_{\beta
  }^{1}},(1+\delta)\sqrt{C_{\beta
}^{2}}],\,\,I_2=\mathbb{R}\setminus I_1.
 \end{split}
\end{equation*}
The case of  $a\in I_2$ yields that 
\begin{equation*}
 |N_{\varepsilon}(a,t)|^2\geq s_1^2\left[\frac{a^2}{t}-(1-\beta t)\right]^2\geq s_1^2C_{\beta}.
\end{equation*}
If $a\in I_1$, $\frac{a^2}{t}-(1-\beta t)$ can be bounded so that we can bound the second term in $ |N_{\varepsilon}(a,t)|^2$, that is 
\begin{equation*}
 \begin{split}
 |N_{\varepsilon}(a,t)|^2&\geq
\left[s_2(\frac{a^2}{t}-(1-\beta t))-\frac{a}{t}\right]^2\\
&=\left[O(s_2)-\frac{a}{t}\right]^2\geq s_1^2C_{\beta}
\end{split}
\end{equation*}
whenever $|\varepsilon|$ is sufficiently small. The above estimates for $|N_{\varepsilon}(a,t)|$ give that
\begin{equation*}
 |N_{\varepsilon}(a,t)|\geq  s_1C_{\beta}.
\end{equation*}
Therefore,
\begin{equation}\label{n-value}
\inf_{a\in \mathbb{R},\,t\in \mathbb{Z}\setminus\{0\}}|N_{\varepsilon}(a,t)|\geq s_1C_{\beta}\geq \sigma C_{\beta,\mu}
\end{equation}
for a positive constant $C_{\beta,\mu}$ depending on $\beta$ and $\mu$, by the domain of $\varepsilon\in \Omega(\sigma,\mu)$. Consequently, for $\widetilde{\mathbf{N}}_{\varepsilon}(k,j)$ defined in \eqref{t-es}, we obtain
\begin{equation}\label{p-bound}
\sup_{k\in \mathbb{Z}^d,\,j\in \mathbb{Z}\setminus\{0\}}|\widetilde{\mathbf{N}}_{\varepsilon}(k,j)|\leq \sup_{a\in \mathbb{R},\,t\in \mathbb{Z}_{+}} |\widetilde{\mathbf{N}}_{\varepsilon}(a,t)|\leq \sigma^{-1}C_{\beta,\mu}.
\end{equation}
It follows  that
\begin{equation*}
\begin{split}
\|\mathcal{N}_{\varepsilon}^{-1}(V)\|_{\rho,m}\leq   \sigma^{-1}C_{\beta,\mu}\|V\|_{\rho,m-2}.
\end{split}
\end{equation*}
This allows us to define
\begin{equation*}
\begin{split}
\left\|\mathcal{N}^{-1}_{\varepsilon} \right\|_{\mathcal{H}^{m-2}\rightarrow \mathcal{H}^{m}}=\sup_{k\in \mathbb{Z}^d,j\in \mathbb{Z}\setminus\{0\}}|\widetilde{\mathbf{N}}_{\varepsilon}(k,j)|.
\end{split}
\end{equation*}
That means $\varepsilon\mathcal{N}^{-1}_{\varepsilon}$ can be bounded from $\mathcal{H}^{m-2}$ to $ \mathcal{H}^{m}$.
\end{proof}

As a matter of fact,  Lemma~\ref{derivative-non} and Lemma~\ref{ppo} give that the operator $\mathcal{T}$ defined in \eqref{plo1} is analytic from the space $\mathcal{H}^{\rho,m}$ to itself.
\begin{remark}
Note that the previous Lemma~\ref{ppo} includes the case of  $\varepsilon\in \mathbb{R}$, which will be used in the later finitely differentiable case (see Lemma~\ref{ppo1}).

Note also that for the  equation \eqref{b-e1}, the nonlinearity will always be regular. Therefore, we just consider the finitely differentiable version with $m>\frac{d+1}{2}$. The analogue of the low regularity results for ODE case would be easier to consider.
\end{remark}
\subsection{Proof of Theorem~\ref{main-p}}\label{sec:mainpr}
In this section, we give the proof of Theorem~\ref{main-p}.
\subsubsection{Regularity in $\varepsilon$}\label{sec:regupde}
Since we want to obtain solutions depending analytically on $\varepsilon$, 
proceeding as in Section~\ref{sec:analyticitysolution}, we consider 
$\mathcal{T}:=\mathcal{T}_{\varepsilon}$  defined in \eqref{plo1} acting on the space $\mathcal{H}^{\rho,m,\Omega}$ consisting of analytic functions of 
$\varepsilon$ taking values in $\mathcal{H}^{\rho,m}$ with $\varepsilon$ ranging over the domain  $\Omega(\sigma, \mu)$. We endow $\mathcal{H}^{\rho,m,\Omega}$ with 
supremum norm
\begin{equation*}
\|U\|_{\rho,m,\Omega}=\sup_{\varepsilon \in \Omega(\sigma, \mu)}\|U_{\varepsilon}\|_{\rho,m},
\end{equation*}
which makes $\mathcal{H}^{\rho,m,\Omega}$ a Banach space. Moreover, it is also a Banach algebra when $m>(d+1)$.
Based on Lemma~\ref{ppo}, we show that the operator $\mathcal{T}$ maps the space $\mathcal{H}^{\rho,m,\Omega}$ into itself. The idea of the proof is similar to Lemma~\ref{diff1}, but the details are different since PDE model \eqref{b-e1} involves a space variable $x$.
\begin{proposition}\label{diff2}
 If $m>(d+5)$, then the operator $\mathcal{T}$ defined in \eqref{plo1} maps the analytic Banach space  $ \mathcal{H}^{\rho,m,\Omega}$ into itself. Precisely, if the mapping $\varepsilon \rightarrow U_{\varepsilon}: \Omega\rightarrow \mathcal{H}^{\rho,m}$ is
  complex differentiable, then,  $\varepsilon\rightarrow \mathcal{T}_{\varepsilon}
 (U_{\varepsilon}): \Omega\rightarrow \mathcal{H}^{\rho,m}$ 
 is also  complex differentiable.  
\end{proposition}
\begin{proof}
From the fixed point equation \eqref{plo1}, we know that $\mathcal{T}_{\varepsilon}$  is composed by $\varepsilon\mathcal{N}^{-1}_{\varepsilon}$ and $h$ defined in Lemma~\ref{derivative-non}. Lemma~\ref{derivative-non} gives that $h(\mathcal{H}^{\rho,m,\Omega}) \subset \mathcal{H}^{\rho,m-2,\Omega}$. Hence, it suffices to verify that  $\varepsilon\mathcal{N}^{-1}_{\varepsilon}(\mathcal{H}^{\rho,m-2,\Omega})\subset \mathcal{H}^{\rho,m,\Omega}$. In the following step, we use a similar method as that in the proof of  Proposition~\ref{diff1}.

 For a fixed $\varepsilon\in \Omega$, we expand $V_{\varepsilon}(\theta,x)\in \mathcal{H}^{\rho,m-2}$ as 
\begin{equation*}
V_{\varepsilon}(\theta,x) =\sum_{k\in\mathbb{Z}^{d},\,j\in\mathbb{Z}\setminus \{0\}}\widehat{V}_{k,j,\varepsilon}e^{\mathrm{i} (k\cdot\theta+j\cdot x)}
\end{equation*}
with 
\begin{equation}\label{coeff-es2}
\begin{split}
\left|\widehat{V}_{k,j,\varepsilon}\right|\leq \left\|V_{\varepsilon}\right\|_{\rho,m-2}e^{-\rho (|k|+|j|)}(|k|^2+|j|^2+1)^{-\frac{m-2}{2}}
\end{split}
\end{equation}
and 
\begin{equation}\label{coeff-es22}
\begin{split}
\left|\frac{d}{d\varepsilon}\widehat{V}_{k,j,\varepsilon}\right|\leq \left\|\frac{d}{d\varepsilon}V_{\varepsilon}\right\|_{\rho,m-2}e^{-\rho (|k|+|j|)}(|k|^2+|j|^2+1)^{-\frac{m-2}{2}}.
\end{split}
\end{equation}
It follows from \eqref{plo} that
\begin{equation*}
\begin{split}
\varepsilon \mathcal{N}_{\varepsilon}^{-1}(V_{\varepsilon})=\sum_{k\in\mathbb{Z}^{d},\,j\in\mathbb{Z}\setminus \{0\}}
\varepsilon \mathbf{N}_{\varepsilon}^{-1}(k\cdot \omega,\,j)\widehat{V}_{k,j,\varepsilon}e^{\mathrm{i} (k\cdot\theta+j\cdot x)},
\end{split}
\end{equation*}
where
\begin{equation*}
\begin{split}
\mathbf{N}_{\varepsilon}^{-1}(k\cdot \omega,\,j)
=\frac{1}{-\varepsilon(k\cdot \omega)^2+\mathrm{i} (k\cdot \omega)
-\varepsilon(\beta j^4-j^2)}=:\mathbf{N}_{\varepsilon}^{-1}.
\end{split}
\end{equation*}
By \eqref{p-bound}, one has
\begin{equation*}
\begin{split}
&\left|\frac{d}{d\varepsilon}\left(\varepsilon \mathbf{N}_{\varepsilon}^{-1}\widehat{V}_{k,j,\varepsilon}\right)\right|\\
&\leq |\mathbf{N}_{\varepsilon}^{-1}|\left|\widehat{V}_{k,j,\varepsilon}\right|+\left|\varepsilon\frac{d}{d\varepsilon} \mathbf{N}_{\varepsilon}^{-1}\right|\left|\widehat{V}_{k,j,\varepsilon}\right|+ \left|\varepsilon \mathbf{N}_{\varepsilon}^{-1}\right|\left|\frac{d}{d\varepsilon}\widehat{V}_{k,j,\varepsilon}\right|\\
&\leq C_{\beta,\mu}\cdot \sigma^{-1}|j|^2\left(\left|\widehat{V}_{k,j,\varepsilon}\right|+\left|\frac{d}{d\varepsilon}\widehat{V}_{k,j,\varepsilon}\right|\right).
\end{split}
\end{equation*}
Together with  \eqref{coeff-es2} and \eqref{coeff-es22}, we get
\begin{equation*}
\begin{split}
&\left\|\frac{d}{d\varepsilon}\left(\varepsilon \mathbf{N}^{-1}_{\varepsilon}\widehat{V}_{k,j,\varepsilon}\right)e^{\mathrm{i} (k\cdot\theta+j\cdot x)}
\right\|_{\rho,\,m-\tau}\\
&\leq C_{\beta,\mu}\cdot\sigma^{-1}|j|^2\left(\left|\widehat{V}_{k,j,\varepsilon}\right|+\left|\frac{d}{d\varepsilon}\widehat{V}_{k,j,\varepsilon}\right|\right)
\|e^{\mathrm{i} (k\cdot\theta+j\cdot x)}\|_{\rho,\,m-\tau}\\
&\leq  C_{\beta,\mu}\cdot\sigma^{-1} |j|^2\left(\|V_{\varepsilon}\|_{\rho,m-2}+\left\|\frac{d}{d\varepsilon}V_{\varepsilon}\right\|_{\rho,m-2}
\right)e^{-\rho (|k|+|j|)}\\
&\ \ \ \ \ \ \ \ \  \ \cdot (|k|^2+|j|^2+1)^{-\frac{m-2}{2}} e^{\rho (|k|+|j|)}(|k|^2+|j|^2+1)^{\frac{m-\tau}{2}}\\
& \leq C_{\beta,\mu}\cdot \sigma^{-1}\left(\|V_{\varepsilon}\|_{\rho,m-2}+\left\|\frac{d}{d\varepsilon}V_{\varepsilon}\right\|_{\rho,m-2}
\right)(|k|^2+|j|^2+1)^{-(\frac{\tau}{2}-2)}.
\end{split}
\end{equation*}
By choosing $d+5<\tau\leq m$,  we obtain that
\begin{equation*}
\begin{split}
\sum_{k\in \mathbb{Z}^d,\,j\in\mathbb{Z}\setminus \{0\}}(|k|^2+|j|^2+1)^{-(\frac{\tau}{2}-2)}\leq C_d \sum_{\kappa=0}^{\infty}(\kappa^2+1)^{-\frac{\tau-d-4}{2}}
< \infty.
\end{split}
\end{equation*}
 As a consequence, by Weierstrass M-test, we conclude that the series
\begin{equation*}
\begin{split}
\sum_{k\in \mathbb{Z}^d,\,j\in\mathbb{Z}\setminus \{0\}} \frac{d}{d\varepsilon}\left(\varepsilon \mathbf{N}^{-1}_{\varepsilon}\widehat{V}_{k,j,\varepsilon}\right)e^{\mathrm{i} (k\cdot\theta+j\cdot x)}
\end{split}
\end{equation*}
converge uniformly on $\varepsilon\in \Omega$ in the space $\mathcal{H}^{\rho,m-\tau}$. Therefore,
\begin{equation*}
\begin{split}
\frac{d}{d\varepsilon}\left(\varepsilon \mathcal{N}_{\varepsilon}^{-1}(V_{\varepsilon})\right)=
\sum_{k\in \mathbb{Z}^d,\,j\in\mathbb{Z}\setminus \{0\}} \frac{d}{d\varepsilon}\left(\varepsilon \mathbf{N}^{-1}_{\varepsilon}\widehat{V}_{k,j,\varepsilon}\right)e^{\mathrm{i} (k\cdot\theta+j\cdot x)}.
\end{split}
\end{equation*}
In conclusion, we have that the map $\varepsilon \rightarrow \varepsilon \mathcal{N}_{\varepsilon}^{-1}(V_{\varepsilon}):\Omega\rightarrow \mathcal{H}^{\rho,m}$  is complex differentiable with derivatives in $\mathcal{H}^{\rho,m-\tau}$ by  $\mathcal{H}^{\rho,m} \subset \mathcal{H}^{\rho,m-\tau}$ and Lemma~\ref{bootstrap} in Appendix.
\end{proof}
\subsubsection{Proof of Theorem~\ref{main-p}}
\label{sec:proofmain-p}
 We now start to deal with the fixed point equation
\begin{equation*}
U(\theta,x)=\mathcal{N}_{\varepsilon}^{-1}\left[ \varepsilon (U^{2})_{xx}+\varepsilon f(\theta,x)\right] \equiv \mathcal{T}(U)(\theta,x).
\end{equation*}
in the space $\mathcal{H}^{\rho,m,\Omega}$. 
We will find a fixed point of $\mathcal{T}$ by considering a small ball $\mathbb{B}_{\mathbf{r}}(0)
\subset \mathcal{H}^{\rho,m,\Omega}$ 
with $C_{\beta,\mu}\cdot \mathbf{r}<\frac{1}{2}$
such that $\mathcal{T}(\mathbb{B}_{\mathbf{r}}(0))\subset \mathbb{B}_{\mathbf{r}}(0)$and $\mathcal{T}$ is a contraction in this ball.

It follows from Lemma~\ref{ppo}  that 
\begin{equation*}
\|\varepsilon \mathcal{N}_{\varepsilon}^{-1}\|_{\rho,m,\Omega}\leq C_{\beta,\mu}.
\end{equation*}
Hence, if $U\in \mathbb{B}_{\mathbf{r}}(0)$, Lemma~\ref{derivative-non} shows that
\begin{equation*}
\begin{split}
\|\mathcal{T}(U)\|_{\rho, m,\Omega}
&\leq\|\mathcal{T}(0)\|_{\rho, m,\Omega}+ \|\mathcal{T}(U)-
\mathcal{T}(0)\|_{\rho, m,\Omega}\\ 
&\leq\|\varepsilon\mathcal{N}_{\varepsilon}^{-1}\|_{\rho, m,\Omega}\left( \|f\|_{\rho, m,\Omega}+
\| Dh(V)U\|_{\rho, m-2,\Omega}\right) \\
&\leq C_{\beta,\mu}\left( \|f\|_{\rho, m,\Omega}+\|V\|_{\rho,m,\Omega}\|U\|_{\rho,m,\Omega}\right)\\
&\leq C_{\beta,\mu}\left( \|f\|_{\rho, m,\Omega}+\mathbf{r}^2\right)
\leq \mathbf{r},
\end{split}
\end{equation*}
provided that we impose the smallness condition for $f$ satisfying 
\begin{equation*}
\|f\|_{\rho, m,\Omega}\leq \frac{\mathbf{r}}{2C_{\beta,\mu}}.
\end{equation*}
Moreover, for $U_1,U_2\in \mathbb{B}_{\mathbf{r}}(0)$, we get that
\begin{equation*}
\begin{split}
\|\mathcal{T}(U_1)-\mathcal{T}(U_2)\|_{\rho, m,\Omega}
&=\|\varepsilon\mathcal{N}_{\varepsilon}^{-1} h(U_1)-\varepsilon\mathcal{N}_{\varepsilon}^{-1} h(U_2)\|_{\rho, m,\Omega}\\
&\leq C_{\beta,\mu}\cdot\mathbf{r}\|U_1-U_2\|_{\rho, m,\Omega}\\
&<\frac{1}{2}\|U_1-U_2\|_{\rho, m,\Omega},
\end{split}
\end{equation*}
which implies that $\mathcal{T}$ is a contraction in the ball $\mathbb{B}_{\mathbf{r}}(0)$.
 In conclusion, there is a unique  fixed point $U$ in the space $\mathcal{H}^{\rho,m,\Omega}$ for equation \eqref{plo1}. Namely, we obtain a  solution $U_{\varepsilon}$  analytic in $\varepsilon$ for equation  \eqref{ffixeq}.
 For $\varepsilon \rightarrow 0$,  we have $\varepsilon\,\rightarrow U_{\varepsilon}$ is continuous, whose  proof is similar to Lemma~\ref{continuous}.
\subsection{Proof of Theorem~\ref{pde-fi}}\label{sec:pde-finitely}
In this section, we consider 
$\mathcal{T}:=\mathcal{T}_{\varepsilon}$  defined in \eqref{plo1} acting on space $\mathcal{H}^{m,\widetilde{\Omega}}$ consisting of differentiable functions of 
$\varepsilon$ taking values in $\mathcal{H}^{m}$ with $\varepsilon$ ranging over the domain  $\widetilde{\Omega}(\sigma, \mu)$. We endow $\mathcal{H}^{m,\widetilde{\Omega}}$ with 
supremum norm
\begin{equation}\label{suprenorm-pde2}
\|U\|_{m,\Omega}=\sup_{\varepsilon \in \widetilde{\Omega}(\sigma, \mu)}\|U_{\varepsilon}\|_{m}.
\end{equation}
We only have the result that the space $\mathcal{H}^{m}$ is a Banach space and it is also a Banach algebra when $m>\frac{d+1}{2}$ but not the space $\mathcal{H}^{m,\widetilde{\Omega}}$ with the supremum norm with respect to $\varepsilon$ defined in \eqref{suprenorm-pde2}.  Consequently, the contraction mapping principle is not enough to get the solution $U_{\varepsilon}$ with optimal regularity in $\varepsilon$. We will combine with the implicit function theorem to obtain the regular solutions.

In order to use the implicit function theorem, analogous to Section ~\ref{sec:finitely},
the main issue is to study the differentiability
of the operator $\mathcal{T}(\varepsilon,U)$ in \eqref{plo1}
considered as an operator from 
$\widetilde{\Omega}\times \mathcal{H}^m$ to $\mathcal{H}^{m}$
as well as the invertibility of 
$D_2\mathcal{T}(\varepsilon, U)$.

  We first present the result with respect to the argument $U$. Since Lemma~\ref{derivative-non} and Lemma~\ref{ppo} also hold in the finitely differentiable setting,  we have the following result when we work in the space $\mathcal{H}^{m}$.
\begin{lemma}\label{ppo1}
	For a fixed $\varepsilon \in \widetilde{\Omega}(\sigma,\mu)$, the operator $\mathcal{T}_{\varepsilon}$ is analytic from the space $\mathcal{H}^{m}$ to itself.  
\end{lemma}

Now, we give the following proposition with the result that the operator $\mathcal{T}$ in \eqref{plo1} is differentiable in the argument $\varepsilon$. Note that $\mathcal{T}$ is composed by $\varepsilon\mathcal{N}^{-1}_{\varepsilon}$ and $h$ defined in \eqref{non-term}. Since $h(\mathcal{H}^{m}) \subset \mathcal{H}^{m-2}$, we  need to  verify that the derivatives of  $\varepsilon\mathcal{N}^{-1}_{\varepsilon}$ with respect to $\varepsilon$ is bounded from the space $\mathcal{H}^{m-2}$ to the space $\mathcal{H}^{m}$.  Similar to Proposition~\ref{derivatives}, we have:
\begin{proposition}\label{derivatives1}
	Fix any $m \in \mathbb{N}$ with $m>\frac{d+1}{2}$ and $\sigma>0$. 
	We consider the map that to every $\varepsilon \in \widetilde{\Omega}$,
	$\varepsilon \mathcal{N}^{-1}_\varepsilon \in B(\mathcal{H}^{m-2}, \mathcal{H}^m)$. Moreover,
	for any $l \in \mathbb{N}$ and $\varepsilon \in \widetilde{\Omega}$, the map  $\varepsilon \rightarrow 
	\varepsilon \mathcal{N}^{-1}_\varepsilon$ is $C^l$ considered as
	a mapping from $\widetilde{\Omega}$ to $B(\mathcal{H}^{m-2}, \mathcal{H}^m)$.  Namely,
	$\frac{d^l}{d \varepsilon^l} (\varepsilon \mathcal{N}_\varepsilon^{-1})\in B(\mathcal{H}^{m-2}, \mathcal{H}^m)$.

	As a matter of fact, something stronger is true. The mapping 
	$\varepsilon \rightarrow \varepsilon  \mathcal{N}^{-1}_\varepsilon$ is 
	real analytic for $\varepsilon \in \widetilde{\Omega}$ and the radius of
	analyticity can be bounded uniformly for all $\varepsilon \in \widetilde{\Omega}$. 
\end{proposition} 

\begin{proof}  
The idea of the proof is similar to Proposition~\ref{derivatives}. Based on the estimates
	$|N_\varepsilon(a,t)| \geq \sigma C_{\beta,\mu}$ in \eqref{n-value} in Lemma~\ref{ppo}, 
	we now expand $N_{\varepsilon+\delta}^{-1}(a,t)$ 
 in powers of  $\delta$ as
	\begin{equation}\label{goodformu1}
	\begin{split} 
N^{-1}_{\varepsilon + \delta}(a,t)  &=\left(-(\varepsilon+\delta)\left[\frac{a^2}{t}-(1-\beta t)\right]+\mathrm{i}\frac{a}{t}\right)^{-1}\\
&= 
	\left( -\varepsilon\left[\frac{a^2}{t}-(1-\beta t)\right]+\mathrm{i}\frac{a}{t}-\delta \left[\frac{a^2}{t}-(1-\beta t)\right] \right)^{-1} \\
	&=\left(-\varepsilon\left[\frac{a^2}{t}-(1-\beta t)\right]+\mathrm{i}\frac{a}{t} \right)^{-1}\left(1-\delta\frac{\left[\frac{a^2}{t}-(1-\beta t)\right]}{-\varepsilon\left[\frac{a^2}{t}-(1-\beta t)\right]+\mathrm{i}\frac{a}{t}}\right)^{-1}. 
	\end{split} 
	\end{equation}
	
By the estimates in Lemma~\ref{ppo}, we observe that the factor 
	$ \frac{\left[\frac{a^2}{t}-(1-\beta t)\right]}{-\varepsilon\left[\frac{a^2}{t}-(1-\beta t)\right]+\mathrm{i}\frac{a}{t}}$
	is bounded uniformly in $a\in \mathbb{R},t\in \mathbb{Z}_{+}$ and 
	$\varepsilon\in \widetilde{\Omega}$.
	
	Therefore, we can expand $\left(1-\delta\frac{\left[\frac{a^2}{t}-(1-\beta t)\right]}{-\varepsilon\left[\frac{a^2}{t}-(1-\beta t)\right]+\mathrm{i}\frac{a}{t}}\right)^{-1} $ in \eqref{goodformu1} in powers of $\delta$ using  the geometric 
	series formula and the radii of convergence are bounded uniformly and 
	the values of the function are also bounded in a ball which is 
	uniform in $a \in \mathbb{R}, t\in \mathbb{Z}_{+}$ and $\varepsilon \in \widetilde{\Omega}$. That means $N_{\varepsilon}^{-1}$ is uniformly analytic in $\varepsilon$ for each  $a \in \mathbb{R}, t\in \mathbb{Z}_{+}$.
	
	In the Fourier space, we know that $\mathcal{N}_{\varepsilon}^{-1}$ is multiplier operator with the multiplier $N_{\varepsilon,k,j}^{-1}$. Precisely, for $\widehat{f}_{k,j}$ being the Fourier coefficients of function $f$ in the space $\mathcal{H}^{m-2}$, the Fourier coefficients $\widehat{(\mathcal{N}^{-1}_{\varepsilon}f)}_{k,j}$ of function $(\mathcal{N}^{-1}_{\varepsilon}f)$ in the space $\mathcal{H}^{m}$ have the structure:
	\begin{equation*}
	\widehat{(\mathcal{N}^{-1}_{\varepsilon}f)}_{k,j}=N^{-1}_{\varepsilon,k,j}\widehat{f}_{k,j}.
	\end{equation*}
	 Hence, we get that $\mathcal{N}^{-1}_{\varepsilon}$ is  analytic in $\varepsilon$.
	 
	Moreover, we  can bound $\| \mathcal{N}_{\varepsilon}^{-1} \|_{\mathcal{H}^{m-2}\rightarrow \mathcal{H}^m}$ by the norm
		defined by
		\begin{equation}\label{sup11}
		\left\|\mathcal{N}^{-1}_{\varepsilon} \right\|_{H^{m-2}\rightarrow \mathcal{H}^m}=\sup_{k\in \mathbb{Z}^{d},\,j\in \mathbb{Z}\setminus\{0\}}\|N^{-1}_{\varepsilon,k,j}\|
		\end{equation}
since the uniform boundness of 
$N^{-1}_{\varepsilon,k,j}$ in $k\in \mathbb{Z}^{d},\,j\in \mathbb{Z}\setminus\{0\}$.  Therefore, when we write $\mathcal{N}^{-1}_{\varepsilon+\delta}=\sum_{n=0}^{\infty}\mathcal{N}^{-1}_{\varepsilon,n}\delta^n$, $\| \mathcal{N}_{\varepsilon,n}^{-1} \|_{\mathcal{H}^{m-2}\rightarrow \mathcal{H}^m}$ can be bounded  by the definition in  \eqref{sup11}. That means $\frac{d^l}{d \varepsilon^l} (\varepsilon \mathcal{N}_\varepsilon^{-1})\in B(\mathcal{H}^{m-2}), \mathcal{H}^{m}$ for every $\varepsilon\in \widetilde{\Omega}$.
\end{proof}

Now, we start to prove Theorem~\ref{pde-fi} by constructing a fixed point $U_{\varepsilon_0}$ for $\varepsilon_0\in \widetilde{\Omega}$ first and then using the implicit function theorem to obtain the optimal regularity of $U_{\varepsilon}$ in $\varepsilon$. It is similar to the proof in Section~\ref{sec:fini}. We omit some details here.
\begin{proof}
	 First, when we choose a small ball $\mathbb{B}_{\mathbf{r}}(0)
	 \subset \mathcal{H}^{m,\widetilde{\Omega}}$, the similar process to Section~\ref{sec:proofmain-p} allows us to obtain a fixed point  $U_{\varepsilon_0}\in\mathcal{H}^{m}$ for some $\varepsilon_0\in \widetilde{\Omega}$ by the contraction argument in this ball.

Then, according to  Lemma~\ref{ppo1} and Proposition~\ref{derivatives1}, we obtain that the operator $\mathcal{T}$ defined on $\widetilde{\Omega}\times \mathcal{H}^{m}$ is $C^l$
	in arguments  $\varepsilon$ and $U$. Namely,
	$\mathbf{T}(\varepsilon,U):=U-\mathcal{T}(\varepsilon,U)$ is $C^l$
	in $\widetilde{\Omega}\times \mathcal{H}^{m}$.
	Based on the first step, we have
	$\mathbf{T}(\varepsilon_{0},U_{\varepsilon_0})=0$. Moreover, $D_2\mathbf{T}(\varepsilon_{0},U_{\varepsilon_0})=Id-D_2\mathcal{T}(\varepsilon_{0},U_{\varepsilon_0})=Id-\varepsilon_0\mathcal{N}_{\varepsilon_0}^{-1}Dh(U_{\varepsilon_0})$ is invertible since $\varepsilon_0\mathcal{N}_{\varepsilon_0}^{-1}Dh(U_{\varepsilon_0})$ is sufficiently small in a small domain of the origin.
	Therefore, by the  implicit function theorem, there exist an open neighborhood included in $\widetilde{\Omega}\times \mathcal{H}^{m}$ of $(U_{\varepsilon_0},\varepsilon_0)$ and 
	a $C^l$ function $U_{\varepsilon}$ 
	satisfying $\mathbf{T}(\varepsilon,U_{\varepsilon})=0$ on this neighborhood.
	\end{proof}

\appendix

\section{Some properties in analytic and finitely differentiable Banach spaces}
\label{sec:appendix} 
\subsection{Analytic functions in Banach space}
\begin{definition}\label{analy-def}
 Let  $X,\,Y$ be complex Banach spaces and $O\subset X$ is open . We say that $f:\,O\rightarrow Y$ is analytic if it is differentiable at all points of 
$O$ and there exists a function $\gamma:=\gamma_{x}(\|z\|)$, with $\frac{\gamma_{x}(\|z\|)}{\|z\|}\rightarrow 0$ 
as $\|z\|\rightarrow 0$, such that
\begin{equation*}
 \|f(x+z)-f(x)-Df(x)\cdot z \| \leq \gamma_{x}(\|z\|)
\end{equation*}
for all $x\in O$ and $z\in X$ such that $(x+z)\in O$.
\end{definition}
Note that Definition~\ref{analy-def} is a rather weak version of differentiability, but it is enough for this paper. For more analyticity of nonlinear functions in Banach spaces, we refer to \cite{Hil57,Muji86}.

The main result of this appendix is the theory of complex analytic functions in Banach space, bootstraping the meaning of derivatives of analytic functions. The result could be deduced from stronger results in \cite{Hil57,ReedS72}, but we thought it would be useful to present a self-contained proof since this lemma could be useful in other applications.
\begin{lemma}\label{bootstrap}
Let $U \subseteq \mathbb{C}$ be open and $X,\,Y$ be complex Banach spaces, $X\subseteq Y$ with continuous embedding. Let $f: U\rightarrow X$, which is differentiable in $Y$ for all $x\in U$, and

\begin{equation}\label{ydiff}
\lim_{h\rightarrow 0}\left\| \frac{f(x+h)-f(x)}{h}-f'(x)\right\|_{Y}=0.
\end{equation}
Then, $f'(x)\in X$ and
\begin{equation}\label{xdiff}
\lim_{h\rightarrow 0}\left\| \frac{f(x+h)-f(x)}{h}-f'(x)\right\|_{X}=0.
\end{equation}
\end{lemma}
We start proving the Cauchy-Goursat theorem for functions satisfying \eqref{ydiff}. The proof is rather straightforward. 
This will lead to a Cauchy formula, from which we can deduce \eqref{xdiff}.

\begin{proposition}\label{prop}
Let $g:U\rightarrow X\subseteq Y$, be differentiable at everywhere in the sense of $Y$ differentiable. Let $\gamma$ be a triangle contour contained in $U$. Then 
\begin{equation*}
\int_{\gamma}g(z)dz=0.	
\end{equation*}
\end{proposition}

Of course, by the usual approximation procedures, one 
can get the result for more general paths. This will not be needed for our purposes. Note that, by the  fact that  $g$ is continuous as a function from $U$ to $Y$, we know that the integrals over paths involved can be understood as Riemann integrals.
\begin{proof}
 Suppose $\gamma$ is a triangular contour with positive orientation, we construct four positively oriented contours that are triangles obtained by  joining the midpoints of the sides of $\gamma$. Then, we have 
\begin{equation*}
\int_{\gamma}g(z)dz=\sum_{i=1}^{4}\int_{\gamma_i}g(z)dz.
\end{equation*}
Let $\gamma_1$ be selected such that
\begin{equation*}
\left| \int_{\gamma}g(z)dz\right| \leq\sum_{i=1}^{4}\left| \int_{\gamma_i}g(z)dz\right| \leq
4\left| \int_{\gamma_1}g(z)dz\right|. 
\end{equation*}
If $\int_{\gamma}g(z)dz=b\neq 0$, we get 
\begin{equation*}
\left| \int_{\gamma_1}g(z)dz\right|\geq \frac{1}{4} |b|.	
\end{equation*}
Proceeding by induction, we get a sequence of triangular contours $\{\gamma_n\}$, whose length equals $2^{-n}|\gamma|$, where $|\gamma|$ denotes the length of $\gamma$, such that
 \begin{equation}\label{lowerbound}
 \left| \int_{\gamma_n}g(z)dz\right|\geq \frac{1}{4^n} |b|.	
 \end{equation}
By the choice of $\gamma_n$, we have 
\begin{equation*}
\overline{Interior\, of\,
	 \gamma_{n+1}}\subset
 \overline{Interior\, of\,
 	\gamma_{n}}
\end{equation*}
and the length of  the sides of $\gamma_n$ goes to $0$ as $n\rightarrow \infty$. Therefore there exists a unique point 
$z_0\in \bigcap_{n}  \overline{Interior\, of\,
	\gamma_{n}}\in U$. 

Since $g$ is differentiable at $z_0$, there is a function $R$ such that 
 \begin{equation*}
 	g(z)=g(z_0)+g'(z_0)(z-z_0)+R(z,z_0),
 \end{equation*}
 where 
 \begin{equation*}
 \|R(z,z_0)\|_{Y} \leq |z-z_0|w(|z-z_0|)
 \end{equation*}
 with $w(|z-z_0|) \rightarrow 0$ when $|z-z_0|\rightarrow 0$.
Integrating $g$ along $\gamma_n$, we find that
\begin{equation*}
\begin{split}
\int_{\gamma_n}g(z)dz&=\int_{\gamma_n}g(z_0)dz+
\int_{\gamma_n}g'(z_0)(z-z_0)dz+\int_{\gamma_n}R(z,z_0)dz\\
&=[g(z_0)-g'(z_0)z_0]\int_{\gamma_n}1dz+
g'(z_0)\int_{\gamma_n}zdz+\int_{\gamma_n}R(z,z_0)dz
\\
&=\int_{\gamma_n}R(z,z_0)dz.
\end{split}
\end{equation*}
Therefore,
\begin{equation}\label{upperbound}
\begin{split}
\left\|\int_{\gamma_n}g(z)dz\right\|_{Y}&\leq
|\gamma_n|\cdot\sup_{z\in \gamma_n}\|R(z,z_0)\|_{Y}\\
&\leq |\gamma_n|\cdot \frac{|\gamma_n|}{2}\cdot
w\left(\frac{|\gamma_n|}{2}\right)\\
&\leq
\frac{|\gamma|^2}{2\cdot 4^n}w\left(\frac{|\gamma_n|}{2}\right)
\end{split}
\end{equation}
by $|z-z_0|<\frac{1}{2}|\gamma_n|$ for $z \in \gamma_n$. Comparing \eqref{lowerbound} and \eqref{upperbound}, we get
$b=0$ as desired.
\end{proof}
As a corollary, we obtain the same conclusion, but assuming only that $g$ is differentiable at all points inside of the triangle except for the center of the small triangles.

 Now we begin to prove Lemma~\ref{bootstrap}. As it is standard, for the function $f$ in Lemma~\ref{bootstrap}, fix $\epsilon$ belonging to interior of $\gamma$, we define 
\begin{equation*}
\begin{split}
g_{\epsilon}(z)=\left\{
\begin{array}{l}
\begin{split}
\frac{f(z)-f(\epsilon)}{z-\epsilon},\,\,z\neq \epsilon,
\end{split}
\\
\\
\begin{split}
f'(z),\,\,z=\epsilon,
\end{split}
\end{array}
\right.
\end{split}
\end{equation*}
which satisfies the hypothesis of the Proposition~\ref{prop} or its corollary. If $\gamma$ is an triangle centered at $\epsilon$, then
\begin{equation*}
0=\int_{\gamma}g_{\epsilon}(z)dz=
\int_{\gamma}\frac{f(z)}{z-\epsilon}dz-f(\epsilon)\int_{\gamma}\frac{1}{z-\epsilon}dz.
\end{equation*}
Hence we satisfy the formula
\begin{equation*}
f(\epsilon)=\frac{1}{2\pi\mathrm{i}}
\int_{\gamma}\frac{f(z)}{z-\epsilon}dz.
\end{equation*}
Now, we can compute the derivative with respect to $\epsilon$ in space $X$
and obtain 
 \begin{equation*}
 f'(\epsilon)=\frac{1}{2\pi\mathrm{i}}
 \int_{\gamma}\frac{f(z)}{(z-\epsilon)^2}dz.
 \end{equation*}
 Of course, since the derivative is obtained as limits of quotients, if the limit exists in $X$, it has to agree with the limit in $Y$.

\subsection{Finitely differentiable functions in Banach space}
For arbitrary Banach spaces $X_1, \cdots, X_i, Y, i\geq 1$, we denote by $A(X^{\otimes i}, Y)$ the space of symmetric continuous $i$-linear
 forms on $X^{\otimes i}:=X_1\times \cdots \times X_i$ 
taking values in $Y$. 
Now we present the converse to Taylor's theorem (see page $6$ in the book \cite{robbin67}).
\begin{lemma}
 Let $O\subset X$ be a convex set and $F:\,O\rightarrow Y,\,f_{i}:\,O\rightarrow A(X^{\otimes i}, Y),\,i=0,\cdots,r.$ For any $x\in O$ and $h\in X$
such that $(x+h)\in O$, we define $R(x,h)$ by 
\begin{equation*}
 F(x+h)=F(x)+\sum_{i=1}^r\frac{f_{i}(x)(h,\cdots,h)}{i!}+R(x,h).
\end{equation*}
If for any $0\leq i\leq r$, $f_{i}$ is continuous and for any $x\in O$, $\frac{\|R(x,h)\|_{Y}}{\|h\|^r_{X}}\rightarrow 0$ as 
$\|h\|^r_{X}\rightarrow 0$, then we say $F$ is of class $C^r$ on $O$ and $D^iF=f_i$ for any $0\leq i\leq r$.

\end{lemma}

\begin{definition}
We denote by $C^{r}(O,Y)$
the space of functions $f:O \rightarrow Y$ with continuous derivatives up to order to $r$. We endow
$C^{r}(O, Y )$ with the norm of the supremum of all the
derivatives. Namely,
\begin{equation}\label{normCr}
\begin{aligned}
\|f\|_{C^{r}}=\max_{0\leq i\leq r}
\sup_{x\in O}|[D^{i}f](x)|_{X^{\otimes i}, Y}
\end{aligned}
\end{equation}
with
\begin{equation*}
\begin{aligned}
|\cdot|_{X^{\otimes i}, Y} \equiv \sup_{\|x_1\|_{X_1}=1,\ldots \|x_i\|_{X_i}=1}
\|A(x_1,\ldots, x_i)\|_Y.
\end{aligned}
\end{equation*}
As it is well known,
the norm \eqref{normCr}
makes $C^{r}(O,Y)$ a Banach space. 
\end{definition} 
\begin{definition}
	We denote by $C^{r+Lip} ( O,  Y )$
	the space of functions in  $C^{r}(O, Y )$
	whose $r-$th derivative is Lipschitz. The Lipschitz constant is
	\begin{equation*}
	\begin{aligned}
	Lip_{O,Y}D^{r}f=\sup_{x_1,\,x_2 \in X \atop x_1\neq x_2}
	\frac{|D^{r}f(x_1)-D^{r}f(x_2)
		|_{X^{\otimes r},Y}}{\|x_1-x_2\|_{X}}.
	\end{aligned}
	\end{equation*}
\end{definition}

We note that since $O$ may be not compact, this definition is
different from the Whitney definition in which the topology is
given by semi-norms of supremum in compact sets. We will not use
the Whitney definition of $C^r$ in this paper.

\begin{definition}
An open set $O$ is called a compensated domain  if   there is a constant $C$ such that
given $x, y \in O$ there is a $C^1$ path $\gamma$  contained in $O$
joining $x,y$ satisfying $|\gamma| \le C \| x - y\|$.
\end{definition}

For $O$ a  compensated domain, we have the mean value theorem
\begin{equation*}
\| f(x) - f(y) \|_Y  \le C \| f\|_{C^1(O, Y)} \| x - y\|_X.
\end{equation*}
   In particular, $C^1$ functions in  a compensated domain are
Lipschitz. It is not difficult to construct non-compensated
domains with $C^1$ functions which are not Lipschitz.

Of course a convex set is compensated and the compensation constant is
$1$. In our paper, we will just be considering domains which are
balls or full spaces. See \cite{Rafael99} for the effects of
the compensation constants in many problems of
the function theory.

\subsection{The standard Sobolev space}
As a matter of fact, we define 
\begin{equation*}
\begin{aligned}
H^m(\mathbb{T}^d):=H^m(\mathbb{T}^d,\mathbb{R}^n):=
\{U=(U_1,\cdots,U_n)|U_i\in H^m(\mathbb{T}^d,\mathbb{R}),\,i=1,\cdots,n\}
\end{aligned}
\end{equation*}
equipped with the norm 
\begin{equation}\label{sobol}
\begin{aligned}
\|U\|_{H^m}=\sum_{0\leq i\leq n}\|U_i\|_{H^m}.
\end{aligned}
\end{equation}
And 
\begin{equation*}
\begin{aligned}
H^m(\mathbb{T}^d,\mathbb{R})=
\{U\in L^2(\mathbb{T}^d,\mathbb{R}):\,D^{|\alpha|} U\in L^2(\mathbb{T}^d,\mathbb{R}), \,\,\,0\leq |\alpha|\leq m\},
\end{aligned}
\end{equation*}
where we use multi-index notation $\alpha=(\alpha_1,\cdots,\alpha_d)\in \mathbb{N}^d$, $|\alpha|=\sum_{i=1}^d \alpha_i$ and  $x=(x_1,\cdots,x_d)\in \mathbb{T}^d$,  $D^{\alpha}:=D^{\alpha}_x=D_{x_1}^{\alpha_1}\cdots D_{x_d}^{\alpha_d}$.
We define 	
\begin{equation*}
\begin{aligned}
\|U\|_{H^m(\mathbb{T}^d,\mathbb{R})}=\sum_{0\leq |\alpha|\leq
	m}\|D^{\alpha} U\|_{L^2}
\end{aligned}
\end{equation*}
with 
\begin{equation*}
\begin{aligned}
\|U\|_{L^2}=\left(\int_{\mathbb{T}^d}|U(\theta)|^2d\theta\right)^{\frac{1}{2}}.
\end{aligned}
\end{equation*}
Indeed, by Fourier transformation, the norm defined in \eqref{sobol} is equivalent to the norm defined by 
Definition~\ref{space} based on the Fourier coefficients. We refer to the book \cite{sobolev, taylor3} for more details.

\bibliographystyle{alpha}
\bibliography{modi-liouvillean-llave}

\def\cprime{$'$}
\begin{thebibliography}{CCCdlL17}

\bibitem[AF03]{sobolev}
Robert~A. Adams and John J.~F. Fournier.
\newblock {\em Sobolev spaces}, volume 140 of {\em Pure and Applied Mathematics
  (Amsterdam)}.
\newblock Elsevier/Academic Press, Amsterdam, second edition, 2003.

\bibitem[AR67]{robbin67}
Ralph Abraham and Joel Robbin.
\newblock {\em Transversal mappings and flows}.
\newblock An appendix by Al Kelley. W. A. Benjamin, Inc., New York-Amsterdam,
  1967.

\bibitem[AZ90]{AZ90}
J\"{u}rgen Appell and Petr~P. Zabrejko.
\newblock {\em Nonlinear superposition operators}, volume~95 of {\em Cambridge
  Tracts in Mathematics}.
\newblock Cambridge University Press, Cambridge, 1990.

\bibitem[BG15]{gazzo15}
Elvise Berchio and Filippo Gazzola.
\newblock The role of aerodynamic forces in a mathematical model for suspension
  bridges.
\newblock {\em Discrete Contin. Dyn. Syst.}, (Dynamical systems, differential
  equations and applications. 10th AIMS Conference. Suppl.):112--121, 2015.

\bibitem[Bou72]{Bou72}
J.~Boussinesq.
\newblock Th\'eorie des ondes et des remous qui se propagent le long d'un canal
  rectangulaire horizontal, en communiquant au liquide contenu dans ce canal
  des vitesses sensiblement pareilles de la surface au fond.
\newblock {\em J. Math. Pures Appl. (2)}, 17:55--108, 1872.

\bibitem[CCCdlL17]{Rafael17}
Renato~C. Calleja, Alessandra Celletti, Livia Corsi, and Rafael de~la Llave.
\newblock Response solutions for quasi-periodically forced, dissipative wave
  equations.
\newblock {\em SIAM J. Math. Anal.}, 49(4):3161--3207, 2017.

\bibitem[CCdlL13]{Rafael13}
Renato~C. Calleja, Alessandra Celletti, and Rafael de~la Llave.
\newblock Construction of response functions in forced strongly dissipative
  systems.
\newblock {\em Discrete Contin. Dyn. Syst.}, 33(10):4411--4433, 2013.

\bibitem[CdlL10]{Cala10}
Renato Calleja and Rafael de~la Llave.
\newblock A numerically accessible criterion for the breakdown of
  quasi-periodic solutions and its rigorous justification.
\newblock {\em Nonlinearity}, 23(9):2029--2058, 2010.

\bibitem[CdlL19a]{Honngyu1}
Hongyu Cheng and Rafael de~la Llave.
\newblock Stable manifolds to bounded solutions in possibly ill-posed pdes,
  2019.
\newblock MP\_ARC \# 19-6.

\bibitem[CdlL19b]{Honngyu2}
Hongyu Cheng and Rafael de~la Llave.
\newblock Time dependent center manifold in pdes, 2019.
\newblock MP\_ARC \# 19-7.

\bibitem[CFG14]{Gentile14}
Livia Corsi, Roberto Feola, and Guido Gentile.
\newblock Convergent series for quasi-periodically forced strongly dissipative
  systems.
\newblock {\em Commun. Contemp. Math.}, 16(3):1350022, 20, 2014.

\bibitem[Die69]{dieu}
J.~Dieudonn\'{e}.
\newblock {\em Foundations of modern analysis}.
\newblock Academic Press, New York-London, 1969.
\newblock Enlarged and corrected printing, Pure and Applied Mathematics, Vol.
  10-I.

\bibitem[dlL09]{Llave09}
Rafael de~la Llave.
\newblock A smooth center manifold theorem which applies to some ill-posed
  partial differential equations with unbounded nonlinearities.
\newblock {\em J. Dynam. Differential Equations}, 21(3):371--415, 2009.

\bibitem[dlLO99]{Rafael99}
R.~de~la Llave and R.~Obaya.
\newblock Regularity of the composition operator in spaces of {H}\"{o}lder
  functions.
\newblock {\em Discrete Contin. Dynam. Systems}, 5(1):157--184, 1999.

\bibitem[dlLS19]{Rafael16}
Rafael de~la Llave and Yannick Sire.
\newblock An a posteriori {KAM} theorem for whiskered tori in {H}amiltonian
  partial differential equations with applications to some ill-posed equations.
\newblock {\em Arch. Ration. Mech. Anal.}, 231(2):971--1044, 2019.

\bibitem[Gaz15]{gazz15}
Filippo Gazzola.
\newblock {\em Mathematical models for suspension bridges}, volume~15 of {\em
  MS\&A. Modeling, Simulation and Applications}.
\newblock Springer, Cham, 2015.
\newblock Nonlinear structural instability.

\bibitem[GBD05]{Gentile05}
Guido Gentile, Michele~V. Bartuccelli, and Jonathan H.~B. Deane.
\newblock Summation of divergent series and {B}orel summability for strongly
  dissipative differential equations with periodic or quasiperiodic forcing
  terms.
\newblock {\em J. Math. Phys.}, 46(6):062704, 20, 2005.

\bibitem[GBD06]{Gentile06}
Guido Gentile, Michele~V. Bartuccelli, and Jonathan H.~B. Deane.
\newblock Quasiperiodic attractors, {B}orel summability and the {B}ryuno
  condition for strongly dissipative systems.
\newblock {\em J. Math. Phys.}, 47(7):072702, 10, 2006.

\bibitem[Gen10a]{Gentile10}
Guido Gentile.
\newblock Quasi-periodic motions in strongly dissipative forced systems.
\newblock {\em Ergodic Theory Dynam. Systems}, 30(5):1457--1469, 2010.

\bibitem[Gen10b]{GG10}
Guido Gentile.
\newblock Quasi-periodic motions in strongly dissipative forced systems.
\newblock {\em Ergodic Theory Dynam. Systems}, 30(5):1457--1469, 2010.

\bibitem[Gen10c]{GGO10}
Guido Gentile.
\newblock Quasiperiodic motions in dynamical systems: review of a
  renormalization group approach.
\newblock {\em J. Math. Phys.}, 51(1):015207, 34, 2010.

\bibitem[GMV17]{Guido17}
Guido Gentile, Alessandro Mazzoccoli, and Faenia Vaia.
\newblock Forced quasi-periodic oscillations in strongly dissipative systems of
  any finite dimension, 2017.

\bibitem[GV17]{Gentile17}
Guido Gentile and Faenia Vaia.
\newblock Response solutions for forced systems with large dissipation and
  arbitrary frequency vectors.
\newblock {\em J. Math. Phys.}, 58(2):022703, 14, 2017.

\bibitem[HP74]{Hil57}
Einar Hille and Ralph~S. Phillips.
\newblock {\em Functional analysis and semi-groups}.
\newblock American Mathematical Society, Providence, R. I., 1974.
\newblock Third printing of the revised edition of 1957, American Mathematical
  Society Colloquium Publications, Vol. XXXI.

\bibitem[IKT13]{kappe03}
H.~Inci, T.~Kappeler, and P.~Topalov.
\newblock On the regularity of the composition of diffeomorphisms.
\newblock {\em Mem. Amer. Math. Soc.}, 226(1062):vi+60, 2013.

\bibitem[KP13]{Krantz}
Steven~G. Krantz and Harold~R. Parks.
\newblock {\em The implicit function theorem}.
\newblock Modern Birkh\"{a}user Classics. Birkh\"{a}user/Springer, New York,
  2013.
\newblock History, theory, and applications, Reprint of the 2003 edition.

\bibitem[KS00]{Kinder00}
David Kinderlehrer and Guido Stampacchia.
\newblock {\em An introduction to variational inequalities and their
  applications}, volume~31 of {\em Classics in Applied Mathematics}.
\newblock Society for Industrial and Applied Mathematics (SIAM), Philadelphia,
  PA, 2000.
\newblock Reprint of the 1980 original.

\bibitem[LS90]{loo}
Lynn~H. Loomis and Shlomo Sternberg.
\newblock {\em Advanced calculus}.
\newblock Jones and Bartlett Publishers, Boston, MA, 1990.

\bibitem[Mar74]{MR74}
Jerry Marsden.
\newblock {\em Applications of global analysis in mathematical physics}.
\newblock Publish or Perish, Inc., Boston, Mass., 1974.
\newblock Mathematical Lecture Series, No. 2.

\bibitem[Muj86]{Muji86}
Jorge Mujica.
\newblock {\em Complex analysis in {B}anach spaces}, volume 120 of {\em
  North-Holland Mathematics Studies}.
\newblock North-Holland Publishing Co., Amsterdam, 1986.
\newblock Holomorphic functions and domains of holomorphy in finite and
  infinite dimensions, Notas de Matem\'{a}tica [Mathematical Notes], 107.

\bibitem[RS75]{ReedS75}
Michael Reed and Barry Simon.
\newblock {\em Methods of modern mathematical physics. {II}. {F}ourier
  analysis, self-adjointness}.
\newblock Academic Press [Harcourt Brace Jovanovich, Publishers], New
  York-London, 1975.

\bibitem[RS80]{ReedS72}
Michael Reed and Barry Simon.
\newblock {\em Methods of modern mathematical physics. {I}}.
\newblock Academic Press, Inc. [Harcourt Brace Jovanovich, Publishers], New
  York, second edition, 1980.
\newblock Functional analysis.

\bibitem[RS96]{Rs96}
Thomas Runst and Winfried Sickel.
\newblock {\em Sobolev spaces of fractional order, {N}emytskij operators, and
  nonlinear partial differential equations}, volume~3 of {\em De Gruyter Series
  in Nonlinear Analysis and Applications}.
\newblock Walter de Gruyter \& Co., Berlin, 1996.

\bibitem[Tay97]{taylor3}
Michael~E. Taylor.
\newblock {\em Partial differential equations. {III}}, volume 117 of {\em
  Applied Mathematical Sciences}.
\newblock Springer-Verlag, New York, 1997.
\newblock Nonlinear equations, Corrected reprint of the 1996 original.

\bibitem[Zeh75]{Zehnder75}
E.~Zehnder.
\newblock Generalized implicit function theorems with applications to some
  small divisor problems. {I}.
\newblock {\em Comm. Pure Appl. Math.}, 28:91--140, 1975.

\end{thebibliography}

\end{document}